\documentclass[11pt]{amsart}
\usepackage{geometry}

\usepackage[mathscr]{euscript}
\usepackage[all]{xy}
\usepackage[T1]{fontenc}
\usepackage{amsfonts}

\usepackage{subcaption}
\usepackage{amssymb}
\usepackage{mathrsfs}
\usepackage[scr=rsfs,cal=boondox]{mathalfa}
\DeclareSymbolFontAlphabet{\mathcalorig}   {symbols}
\usepackage{xfrac}
\usepackage{multicol}
\usepackage{multirow}
\usepackage[dvipsnames,usenames]{xcolor}
\usepackage[colorlinks=true,urlcolor=blue,linkcolor=blue,citecolor=blue]{hyperref}
\usepackage{upgreek}
\hypersetup{
	linkbordercolor={1 0 0}, 
	citebordercolor={0 1 0} 
}
\usepackage{cite}

\usepackage{tikz}

\usepackage{graphicx}
\usepackage{float}
\usepackage{calc}
\usepackage[size=footnotesize]{caption}
\usepackage{subcaption}

\graphicspath{{graphics/}}

\newcommand{\al}{\alpha}
\newcommand{\be}{\beta}
\newcommand{\ga}{\gamma}

\newcommand{\si}{\sigma}

\newcommand{\La}{\Lambda}

\newcommand{\De}{\Delta}
\newcommand{\Si}{\Sigma}
\newcommand{\ZZ}{{\mathbb Z}}

\newcommand{\QQ}{{\mathbb Q}}
\newcommand{\RR}{{\mathbb R}}

\newcommand{\cC}{\mathcal C}

\newcommand{\cF}{\mathcal F}
\newcommand{\cG}{\mathcal G}

\newcommand{\cL}{\mathscr L}

\newcommand{\cX}{\mathcalorig X}

\newcommand{\tr}{\mathsf{T}}
\newcommand{\lto}{\longrightarrow}
\newcommand{\lk}{\operatorname{\ell{\it k}}}
\newcommand{\sig}{\operatorname{sig}}
\newcommand{\sm}{\smallsetminus}
\newcommand{\co}{\colon}
\newcommand{\spans}{\operatorname{span}}
\newcommand{\In}{\operatorname{In}}
\newcommand{\Out}{\operatorname{Out}}

\newcommand{\ft}{{\mathfrak{t}}}
\newcommand{\wt}{\widetilde}

\newcommand{\im}{\operatorname{im}}

\newcommand{\Char}{\operatorname{Char}}
\newcommand{\Short}{\operatorname{Short}}
\newcommand{\Spin}{\operatorname{Spin}}
\newcommand{\spin}{\operatorname{spin}}
\newcommand{\flows}{\mathcal F}
\newcommand{\cuts}{\mathcal C}

\newcommand{\inject}{\hookrightarrow}

\newcommand*\wbar[1]{
  \hbox{ \kern-0.2em%
    \vbox{%
      \hrule height 0.75pt  
      \kern0.25ex
      \hbox{%
        \kern-0.10em
        \ensuremath{#1}%
        \kern-0.05em
      }%
    }%
  \kern0.05em}%
}
\makeatletter
\newcommand*\bigcdot{\mathpalette\bigcdot@{.55}}
\newcommand*\bigcdot@[2]{\mathbin{\vcenter{\hbox{\scalebox{#2}{$\m@th#1\bullet$}}}}}
\makeatother

\newtheorem{theorem}{Theorem} [section]
\newtheorem{lemma}[theorem]{Lemma}

\newtheorem{corollary}[theorem]{Corollary}
\newtheorem{question}[theorem]{Question}

\theoremstyle{definition}     
\newtheorem{definition}[theorem]{Definition}

\theoremstyle{remark}
\newtheorem{remark}[theorem]{Remark}
\newtheorem{example}[theorem]{Example}

\title[{\it Mutation, surface graphs, and alternating links in surfaces}]
{Mutation, surface graphs, and \\ alternating links in surfaces}

\author[Boden]{Hans U. Boden}
\address{Mathematics \& Statistics, McMaster University, Hamilton, Ontario Canada}
\email{boden@mcmaster.ca}
\thanks{HB was partially funded by the Natural Sciences and Engineering Research Council of Canada. This research was made possible by travel support from the Sydney Mathematical Research Institute at the University of Sydney, and the University of Sydney Faculty of Science Visiting Scholar Scheme.} 

\author[Dancso]{Zsuzsanna Dancso}
\address{Mathematics \& Statistics, University of Sydney, Sydney, NSW Australia}
\email{zsuzsanna.dancso@sydney.edu.au}

\author[Lin]{Damian J. Lin}
\address{Mathematics \& Statistics, University of Sydney, Sydney, NSW Australia}
\email{d.lin@maths.usyd.edu.au}

\author[Wilkinson-Finch]{Tilda S. Wilkinson-Finch}
\address{Mathematics \& Statistics, University of Sydney, Sydney, NSW Australia}
\email{twil0190@uni.sydney.edu.au}

\subjclass[2020]{57M15 (primary), 05C21, 05C10, 11H55, 57K10, 57K12 (secondary)}
\keywords{alternating link, Tait graph, definite lattice, integer cuts and flows, Tait conjectures, mutation, $d$-invariants}

\begin{document}

\begin{abstract}
In this paper, we study alternating links in thickened surfaces in terms of the lattices of integer flows on their Tait graphs. We use this approach to give a short proof of the first two generalised Tait conjectures. We also prove that the flow lattice is an invariant of alternating links in thickened surfaces and is further invariant under disc mutation. For classical links, the flow lattice and $d$-invariants are complete invariants of the mutation class of an alternating link. For  links in thickened surfaces, we show that this is no longer the case by finding a stronger mutation invariant, namely the Gordon-Litherland linking form. In particular, we find alternating knots in thickened surfaces which have isometric flow lattices but with non-isomorphic linking forms. 
\end{abstract}

\maketitle

 
\section{Introduction}
This paper investigates the lattice of integer flows of the Tait graphs of a link diagram, and their $d$-invariants, as mutation invariants of alternating links in thickened surfaces. 

Our motivation is to generalise to the setting of virtual links, insofar as possible, Greene's results classifying alternating links up to mutation in terms of their $d$-invariants \cite{Greene-2011}. To this end, we study mutation invariants of alternating links in thickened surfaces. We establish that the lattices of integer flows of the Tait graphs are mutation invariants of alternating surface links, and thus so are their $d$-invariants. We show that a range of other virtual link invariants -- such as the mock Alexander polynomials and signatures -- are also invariants of mutation. We use this result to prove that, in contrast to the classical case, for alternating links in thickened surfaces, the $d$-invariants and lattices of integer flows are not complete invariants of the mutation class. As a further application, we give a short proof for the first and second generalised Tait conjectures for non-split alternating links in thickened surfaces, following the approach of \cite{Greene-2017}.
For other recent results on the generalised Tait conjectures, see \cite{Boden-Karimi-2019, Boden-Karimi-Sikora, Kindred-2022}.

\vspace*{3mm}  \noindent
{\bf Conway mutation and branched double covers.} Conway mutation is the operation of cutting along a 2-sphere $S$ that meets the link $L$ transversely in four points, and regluing after performing an involution that permutes the four points of $S \cap L$. Two links are called \textit{mutants} if they can be related by a finite sequence of mutations. For classical links, most invariants (such as the Alexander, Jones, and HOMFLY polynomials, and the signature) are unable to distinguish mutant links.  Indeed, positive mutant knots have $S$-equivalent Seifert matrices \cite[Theorem 2.1]{Kirk-Livingston-2001}. Positive mutants were also conjectured to be concordant \cite[Problem 1.53]{Kirby-Problem}, though counterexamples to this conjecture were found in \cite{Kirk-Livingston-1999, Kirk-Livingston-2001}.   In \cite{Traldi-2022},  Traldi proves a general result classifying links up to mutation in terms of their Goeritz matrices.

Branched covering spaces of links in $S^3$ have been a rich source of link invariants. Among them, the branched double cover $X_2(L)$ plays the preeminent role.  For example, when $L$ is a two-bridge knot or link, the branched double cover $X_2(L)$ is a lens space. Schubert gave a complete classification of two-bridge knots and links in terms of the topological classification of lens spaces \cite{Brody-1960}.  In \cite{Greene-2011}, Greene proved a highly non-trivial generalisation of Shubert's theorem to alternating links.
Mutation enters the picture since a theorem of Viro asserts that if $L$ and $L'$ are mutant links, then $X_2(L) \cong X_2(L')$ \cite{Viro-1976}. Thus the homeomorphism type of $X_2(L)$ is also unable to distinguish mutant links.

\vspace*{3mm}  \noindent
{\bf Greene's lattice-theoretic approach.} In \cite[Theorem 1.1]{Greene-2011}, Greene proved a converse to Viro's theorem: two non-split alternating links $L$ and $L'$ are mutants if and only if their branched double covers $X_2(L)$ and $X_2(L')$ are homeomorphic. Greene also identified a complete invariant of the homeomorphism type of the branched double cover $X_2(L)$, namely, the Heegaard Floer homology groups $\widehat{HF}(X_2(L))$, or equivalently, the $d$-invariants. The $d$-invariants were originally defined by Ozsv\'ath and Szab\'o \cite{Ozsvath-Szabo-2003a} and can be described as follows. For non-split alternating links, the branched double cover $X_2(L)$ is an $L$-space, which means that the Heegaard Floer homology group $\widehat{HF}(X_2(L),\ft)$ has rank one for each $\spin^c$ structure $\ft$. The  $d$-invariant $d(X_2(L),\ft) \in \QQ$ records the absolute grading of a preferred generator. Thus, $\widehat{HF}(X_2(L))$ is completely determined by the $d$-invariant map $d\co \Spin^c(X_2(L)) \lto \QQ$, and by Greene's Theorem they give a complete set of invariants for the homeomorphism type of  $X_2(L)$.

For alternating links, there is a purely lattice-theoretic definition of the $d$-invariants  \cite{Greene-2011} (we review this in detail in Section \ref{subsec-dInv}). The lattice theoretic $d$-invariant, $d_{\cF}$, is constructed from the lattice of integer flows $\cF(G)$ of a Tait graph $G$ of an alternating link diagram, and it can be identified with the Heegaard Floer $d$-invariant. Greene showed that $d_\cF$ is a complete invariant of the flow lattice up to isometry, and a complete invariant of the graph $G$ up to 2-isomorphism \cite[Theorem 1.3]{Greene-2011}. 

In turn, Greene proves \cite[Section 3.4]{Greene-2011} that the 2-isomorphism class of the Tait graph uniquely determines the mutation class of $L$. This is established by studying planar embeddings of two-edge-connected planar graphs and using a classical theorem of Whitney to show that any two planar embeddings of 2-isomorphic graphs differ by a finite sequence of basic moves (flips, planar switches, swaps, and isotopies). Each of these moves can be translated into either a mutation or a planar isotopy on the link diagram $D$.  Therefore, $d_{\cF}$ -- and equivalently, the Heegaard-Floer $d$-invariant -- is a complete mutation invariant of alternating links.

In a subsequent paper \cite{Greene-2017}, Greene gave a geometric characterisation of alternating links: a non-split link $L$ is alternating if and only if there exist a pair of spanning surfaces for $L$, which are positive and negative definite with respect to the Gordon-Litherland pairing. A similar result was also obtained by Howie \cite{Howie-2017}. As an application, also exploiting the equivalence between mutation of alternating knots and 2-isomorphisms of the Tait graphs, Greene gave a geometric proof for the first two Tait conjectures, namely that any two connected reduced alternating diagrams of the same link have equal crossing number and writhe.

\vspace*{3mm} \noindent
{\bf Summary of results.} In this paper, we study analogous questions for alternating links in thickened surfaces.  Given an alternating link diagram in a closed surface $\Si$, we use the Tait graphs embedded in $\Si$ to obtain lattices of integer flows $\cF$ and corresponding $d$-invariants.

For planar graphs, there is a classical duality theorem \cite{Backer-et-al} for the cut ($\cC$) and flow lattices: if $G$ and $G^*$ are planar duals, then $\cC(G^*)\cong \cF(G)$. In Theorem~\ref{thm:SurfaceCutFlowDuality} we prove the analogous theorem for graphs $G$ and $G^*$ which are duals on a closed surface $\Si$: there is a short exact sequence  
$
0\to \cC(G^*) \to \cF(G) \to H_1(\Si;\ZZ)\to 0.
$
We deduce a \textit{chromatic duality} theorem for $d$-invariants (Theorem~\ref{thm-ChromDuality}). 

The Gordon-Litherland pairing was extended to links in thickened surfaces in \cite{Boden-Chrisman-Karimi}. Greene's geometric characterisation of alternating links \cite{Greene-2017} was also generalised to alternating links in thickened surfaces in \cite{Boden-Karimi-2023b}. In the present paper we show that the flow lattice $\cF(G)$ of a Tait graph $G$ associated to a checkerboard coloured reduced alternating link diagram coincides with the Gordon-Litherland pairing on the first homology of the corresponding checkerboard surface (Theorem~\ref{thm-equiv}).  We then combine these results with the Discrete Torelli theorem \cite{Caporaso-Viviani,Su-Wagner} -- a completeness result for lattices of integer flows -- to give a short proof of the first two generalised Tait conjectures (Theorem~\ref{Thm-Tait}). Namely, any two reduced, alternating diagrams of a non-split link in a thickened surface have the same crossing number and writhe. 

In Theorem~\ref{thm-isom}, we show that the Gordon-Litherland pairings on the positive and negative definite checkerboard surfaces are invariants of non-split alternating links. Thus, so are lattices of integer flows of the Tait graphs, and the corresponding $d$-invariants (Corollary~\ref{cor-invariance}).
Furthermore, we show that the flow lattices and $d$-invariants are, in fact, invariant under disc mutation of alternating links in thickened surfaces (Theorem \ref{thm-mutation1}). 

For classical alternating links, Greene \cite{Greene-2011} proved that the $d$-invariants are complete invariants of the mutation class, and it is natural to ask whether the same is true for alternating links in thickened surfaces. However, the Gordon-Litherland pairing -- which coincides with the flow lattice -- is a symmetrisation of the \textit{Gordon-Litherland linking form} \cite{Boden-Karimi-2023b}, which in turn gives rise to a range of invariants potentially stronger than the $d$-invariants and flow lattices. In Theorem~\ref{thm-linking-form-mutation}, we show that the Gordon-Litherland linking form is indeed unchanged under disc mutation of surface link diagrams, and as a result so is any invariant derived from it, including the mock Alexander polynomial and mock Levine-Tristram signatures.

Finally, in Example~\ref{ex-NonCompleteness} and Corollary~\ref{cor-NonCompleteness}, we show that -- in contrast to the classical setting -- the $d$-invariants and lattices of integer flows are not complete invariants of alternating surface links up to disc mutation. This is achieved by finding an example of non-mutant alternating knots with isometric Gordon-Litherland pairings, and in particular different mock Alexander polynomials. 

One motivation for this paper is to extend the methods and results of \cite{Greene-2011} to virtual links. Virtual links can be viewed as links in thickened surfaces up to homeomorphisms and stabilisation, and by Kuperberg’s theorem, \cite{Kuperberg}, any minimal genus realisation of a virtual link is unique up to homeomorphism. The cut and flow lattices and $d$-invariants studied here are invariant under surface homeomorphism, and any reduced alternating diagram of a link in a surface is minimal genus \cite[Corollary 3.6]{Boden-Karimi-2023}. Therefore, the results in this paper can be restated in the language of virtual links.

\vspace*{3mm}  \noindent
{\bf Organisation.} The paper is organised as follows. In Section \ref{section-LinksSurfacesGraphs}, we introduce basic notions for links in thickened surfaces, surface graphs, and mutation. In Section \ref{section-Lattices}, we introduce the flow and cut lattices for graphs in surfaces, define their $d$-invariants, and prove a duality result for the  $d$-invariants under surface duality.  In Section \ref{section-GLform}, we recall the Gordon-Litherland linking form and pairing, and we relate the Gordon-Litherland pairing to the lattice of flows for alternating links in thickened surfaces. Section \ref{section-main} contains the main results, including a new proof of the first two Tait conjectures for links in thickened surfaces and proofs that the Gordon-Litherland pairing and linking form are mutation invariants of alternating links in thickened surfaces. It also presents an example of two knots that are not mutants but whose lattices of cuts and flows are isometric. It follows that the cut and flow lattices, and $d$-invariants, are not complete invariants of mutation type.

\vspace*{3mm}  \noindent
{\bf Conventions.} 
Spanning surfaces are assumed to be compact and connected but not necessarily orientable. Graphs are multigraphs and may contain loops and multiple edges. Decimal numbers such as 3.7 and 5.2429 refer to virtual knots in Green's tabulation \cite{Green}.

\section{Links, surfaces and graphs} \label{section-LinksSurfacesGraphs}
In this section, we introduce some basic notions for links in thickened surfaces, such as link diagrams, equivalence, stabilization, checkerboard colourability, Tait graphs, and mutation.

\subsection{Links in thickened surfaces}\label{subsec:Links}
Let $\Si$ be a closed, oriented surface and  $I = [0,1]$  the unit interval.  A \textit{link} $L$ in $\Si \times I$ is an embedding $L\co \sqcup_{i=1}^m S^1 \hookrightarrow \Si \times I$, taken up to isotopy and orientation preserving homeomorphisms of the pair $(\Si\times I, \Si \times \{0\}).$ Two links $L$ and $L'$ in $\Si\times I$ are \textit{equivalent} if there is a homeomorphism of  $(\Si\times I, \Si\times \{0\})$ taking $L$ to $L'$.  

A \textit{diagram} $D$ for the link is the image of $L \subseteq \Si \times I$ under regular projection  $\Si \times I \to \Si$, with finitely many double points (or \textit{crossings}). Crossings are drawn in the usual way with the undercrossing arc broken into two arcs at each crossing, see Figure~\ref{fig:surface-knot-diagram-examples}. Two link diagrams represent equivalent links if and only if they are equivalent via Reidemeister moves and homeomorphisms on the surface.

\begin{figure}[hbt!]
	\centering
	\hspace*{\fill}
	\begin{subfigure}[t]{0.3 \textwidth}
		\centering
        \def\svgwidth{0.9\columnwidth}
		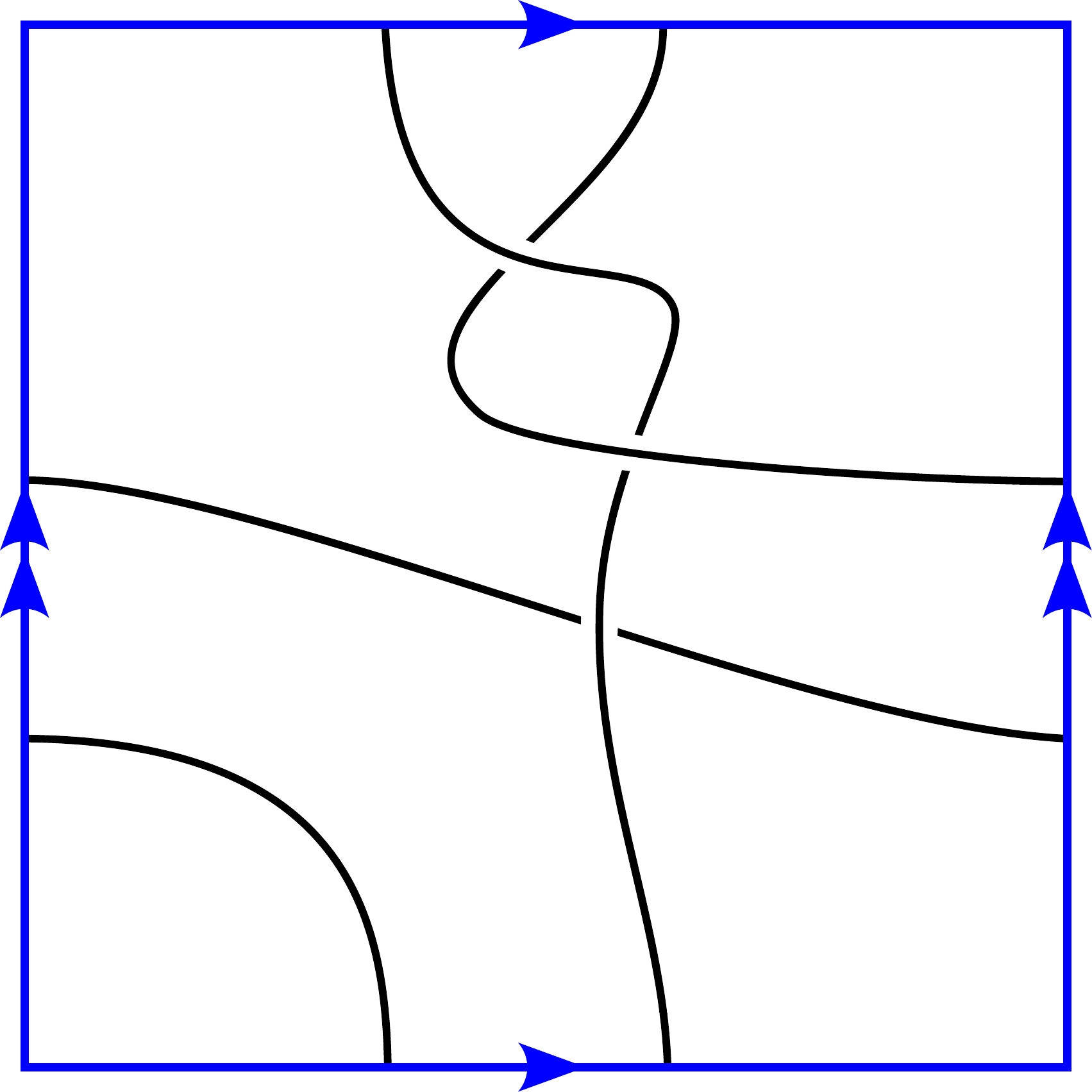
		\caption{$K_{3.7}$}
		\label{fig:3-7-gluing}
	\end{subfigure}
	\hspace*{\fill}
	\begin{subfigure}[t]{0.3 \textwidth}
		\centering
		\def\svgwidth{0.9\columnwidth}
		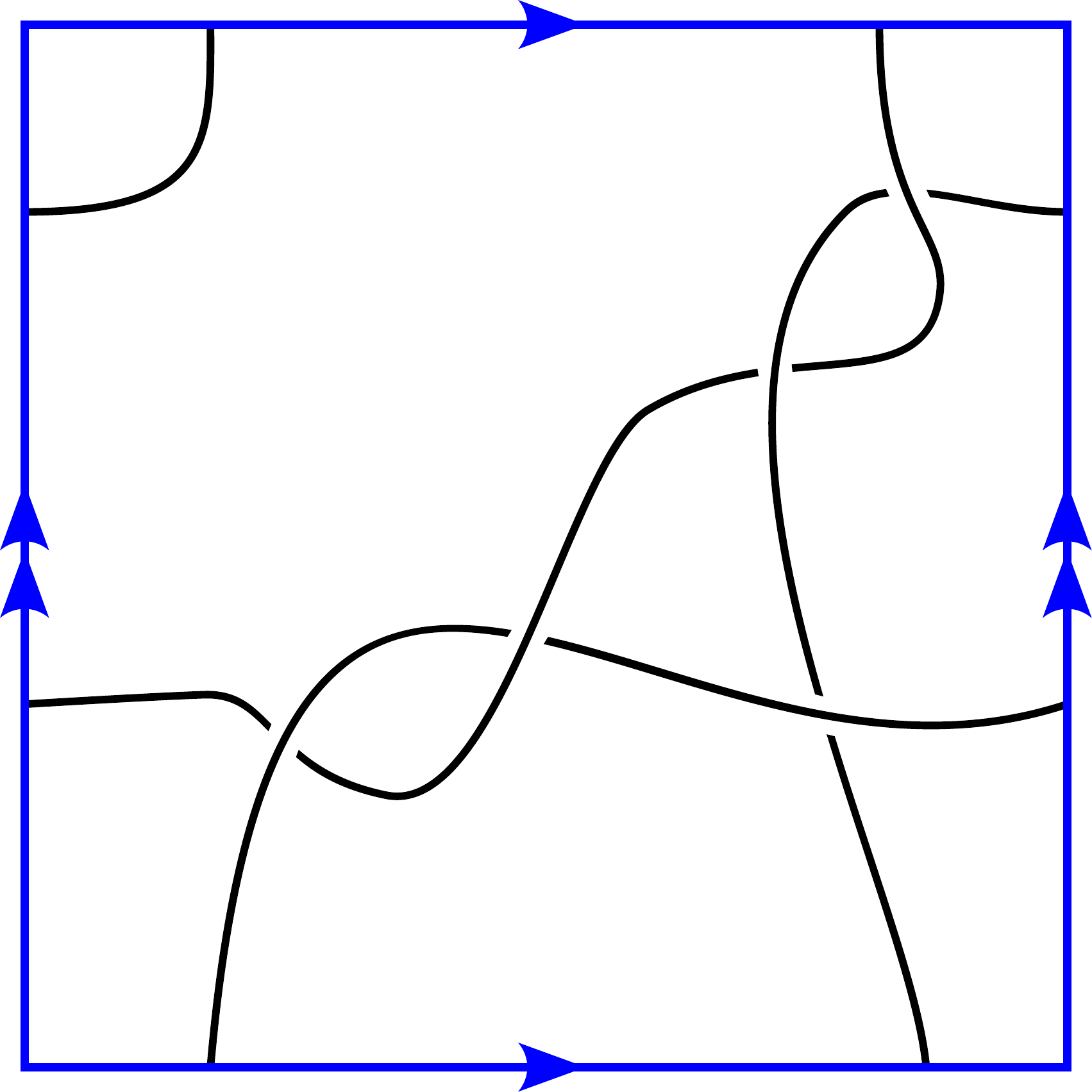
		\caption{$K_{5.2429}$}
		\label{fig:5-2429-gluing}
	\end{subfigure}
    \hspace*{\fill}
	\caption{Two knot diagrams on a torus.}
	\label{fig:surface-knot-diagram-examples}
\end{figure}

A link  $L \subseteq \Si \times I$ is said to be \textit{split} if it can be represented by a disconnected diagram. A link diagram $D \subseteq \Si$ is said to be  \textit{cellularly embedded}, or simply \textit{cellular}, if $\Si \sm D$ is a union of disks. Note that if $D$ is cellularly embedded, then it is necessarily connected.

Virtual links are links in thickened surfaces of any genus modulo the equivalence relation of \textit{stable equivalence}. We define this at the level of surface link diagrams, but it can be stated equivalently for embeddings $L \inject \Si \times I$. 

Intuitively, given a surface link diagram $D \subset\Si$, \textit{stabilisation}  `adds a handle' to $\Si$ which does not interact with $D$. Specifically, if $B_{1}$ and $B_{2}$ are two 2-disks in $\Si$ disjoint from $D$ and each other, we can remove $B_{1}$ and $B_{2}$ from $\Si$, and attach a 1-handle with boundary $\partial B_{1} \sqcup \partial B_{2}$.  The inverse operation is called \textit{destabilisation}, and intuitively it `removes an empty handle' from $\Si$. Specifically, if $\ga \subseteq \Si$ is an essential simple closed curve disjoint from $D$, we can cut the surface $\Si$ along $\ga$ and cap off the two boundary components by attaching 2-disks.

Two links $L \subseteq \Si \times I$ and $L' \subseteq \Si'\times I$ are said to be \textit{stably equivalent} if they are related by a finite sequence of isotopies, homeomorphisms, stabilisations, and destabilisations. A \textit{virtual link} is a stable equivalence class of surface link diagrams. 

A link diagram $D$ in $\Si \times I$ is \textit{alternating} if following each link component along, the crossings alternate over and under. A virtual link $L$ is alternating if it possesses an alternating diagram on some closed surface $\Si$. An alternating diagram is \textit{reduced} if it does not include any \textit{nugatory crossings}. A crossing $c$ of a diagram $D \subseteq \Si$ is nugatory if there exists a simple closed curve $\gamma$ on $\Si$, which intersects the diagram only at $c$. A nugatory crossing is \textit{removable} if $\gamma$ bounds a disc: in this case, the crossing can be removed by un-twisting the part of $D$ contained in the disc. A nugatory crossing is \textit{essential} (non-removable) if the curve $\gamma$ is essential.


\subsection{Checkerboard colourings and Tait graphs}\label{subsec-CCandTaitGraph}
A link diagram $D \subseteq \Si$ is  \textit{checkerboard colourable} if the regions of $\Si \sm D$ can be coloured black and white, such that adjacent regions have opposite colours. A link $L \subseteq \Si \times I$ is \textit{checkerboard colourable} if it admits a checkerboard colourable diagram.

In \cite{Kamada-2002}, Kamada showed that every cellularly embedded alternating link diagram is checkerboard colourable. She further proved that a cellularly embedded link diagram $D \subseteq \Si$ is checkerboard colourable if and only if it can be converted into an alternating diagram by a finite sequence of crossing changes (see Lemma 7, \cite{Kamada-2002}).

In a checkerboard coloured link diagram, each crossing locally appears in one of two ways -- type A or type B -- as in Figure~\ref{fig:crossing-type}. In an alternating diagram, all crossings have the same type. In fact, a link diagram is alternating if and only if it admits a checkerboard colouring so that all crossings have the same type. We adopt the convention whereby alternating link diagrams $D\subseteq \Si$ are checkerboard coloured so that all crossings have type A.

\begin{figure}[hbt]
\centering
\includegraphics[height=24mm]{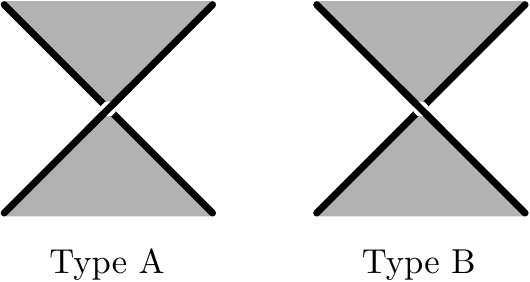} \qquad \qquad \includegraphics[height=24mm]{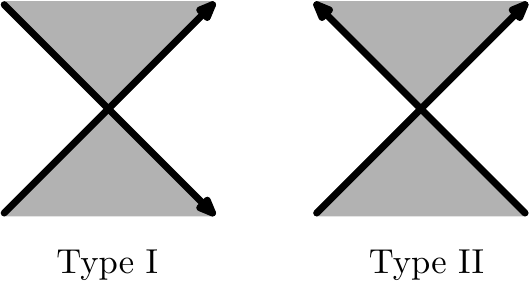} 
\caption{On left, crossings of types A and B. On right, crossings of types I and II.}
\label{fig:crossing-type}
\end{figure}

To any link $L$ in $\Si \times I$ admitting a cellular, checkerboard colourable diagram, we associate two checkerboard surfaces, one black and the other white. We denote them by $F_b$ and $F_w$; they are compact unoriented surfaces in $\Si \times I$ with boundary $L$. The black surface $F_b$ is constructed by joining the black regions with half-twisted bands at each crossing, and the white surface $F_w$ is obtained from the white regions in the same way.

Given a diagram $D\subseteq \Si$ with black and white checkerboard surfaces, the \textit{Tait graphs} are defined as follows. The Tait graph for the black surface is denoted $G_b(D)$ and is the surface graph in $\Si$ with one vertex in each black region and one edge for each crossing. An edge connects two vertices whenever their associated black regions meet at a crossing: see Figure~\ref{fig:5-2429-taits}.  The Tait graph for the white surface is denoted $G_w(D)$ and constructed similarly.  Note that $G_b(D)$ and $G_w(D)$ are deformation retracts of $F_b$ and $F_w$, respectively.  

A graph $G$ embedded in $\Si$ is said to be \textit{cellularly embedded}, or \textit{cellular}, if the regions of $\Si \sm G$ are all disks. If the link diagram $D$ is cellularly embedded and checkerboard coloured, then its Tait graphs $G_b(D)$ and $G_w(D)$ are also cellularly embedded. We will see in Section~\ref{subsec-Flows} that the black and white Tait graphs $G_b(D)$ and $G_w(D)$ are in fact \textit{surface dual} to each other in $\Si$.

\begin{example}\label{ex-Tait} Let $K = 5.{2429}$, and $D$ the surface knot diagram given in Figure~\ref{fig:5-2429-vknot}. The Tait graphs are given abstractly in Figure~\ref{fig:5-2429-taits-abstract}. We will carry $K_{5.2429}$ and the corresponding Tait graphs throughout the paper as a running example.
\hfill $\Diamond$
\end{example}

\begin{figure}[hbt]
	\hspace*{\fill}
	\begin{subfigure}[b]{0.3 \textwidth}
		\centering
		\def\svgwidth{\columnwidth}
\begingroup%
  \makeatletter%
  \providecommand\color[2][]{%
    \errmessage{(Inkscape) Color is used for the text in Inkscape, but the package 'color.sty' is not loaded}%
    \renewcommand\color[2][]{}%
  }%
  \providecommand\transparent[1]{%
    \errmessage{(Inkscape) Transparency is used (non-zero) for the text in Inkscape, but the package 'transparent.sty' is not loaded}%
    \renewcommand\transparent[1]{}%
  }%
  \providecommand\rotatebox[2]{#2}%
  \newcommand*\fsize{\dimexpr\f@size pt\relax}%
  \newcommand*\lineheight[1]{\fontsize{\fsize}{#1\fsize}\selectfont}%
  \ifx\svgwidth\undefined%
    \setlength{\unitlength}{490.33411148bp}%
    \ifx\svgscale\undefined%
      \relax%
    \else%
      \setlength{\unitlength}{\unitlength * \real{\svgscale}}%
    \fi%
  \else%
    \setlength{\unitlength}{\svgwidth}%
  \fi%
  \global\let\svgwidth\undefined%
  \global\let\svgscale\undefined%
  \makeatother%
  \begin{picture}(1,1.00000418)%
    \lineheight{1}%
    \setlength\tabcolsep{0pt}%
    \put(0,0){\includegraphics[width=\unitlength,page=1]{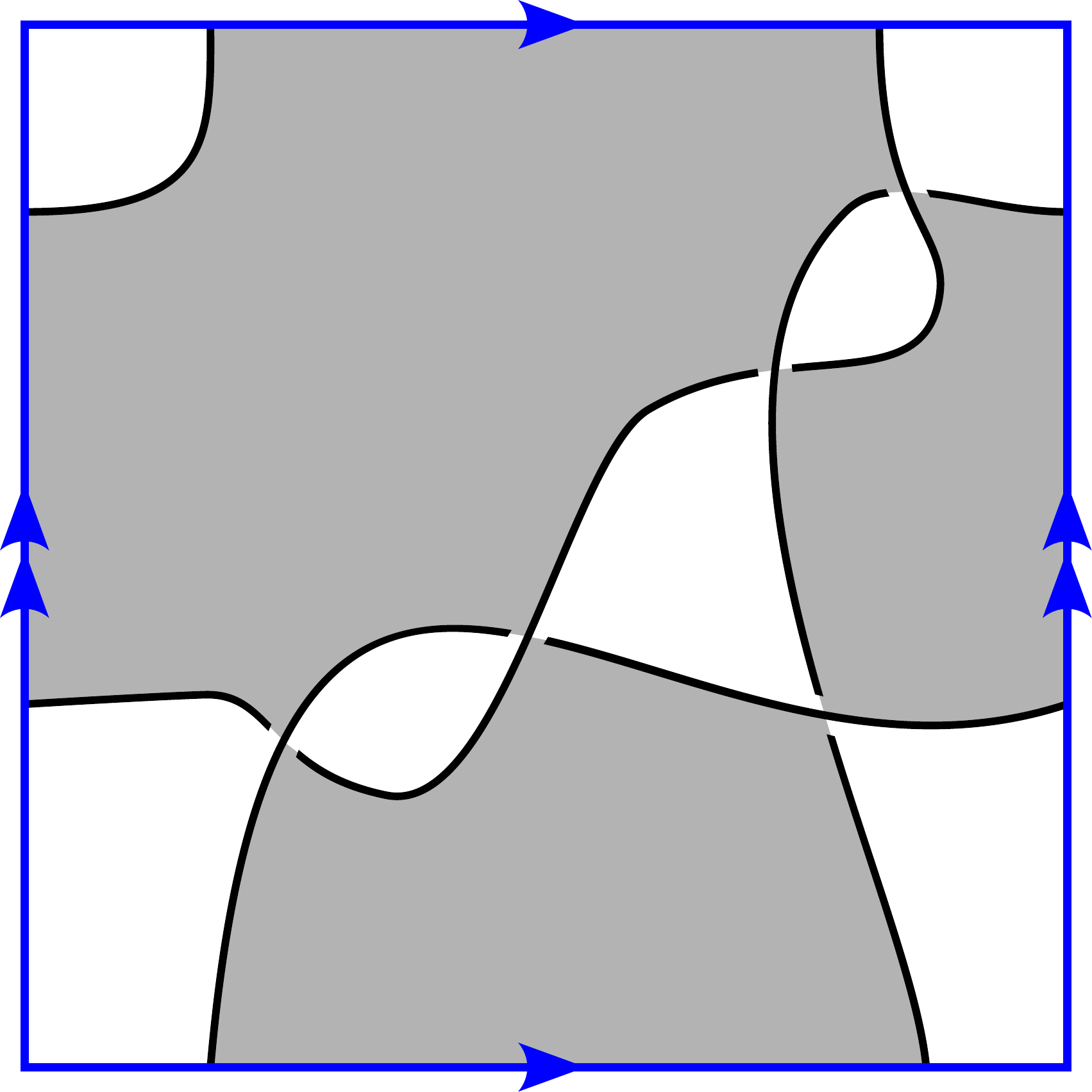}}%
    \put(0.19111203,0.27473241){\makebox(0,0)[lt]{\lineheight{1.25}\smash{\begin{tabular}[t]{l}1\end{tabular}}}}%
    \put(0.51601701,0.41942285){\makebox(0,0)[lt]{\lineheight{1.25}\smash{\begin{tabular}[t]{l}2\end{tabular}}}}%
    \put(0.78926241,0.26542761){\makebox(0,0)[lt]{\lineheight{1.25}\smash{\begin{tabular}[t]{l}3\end{tabular}}}}%
    \put(0.72845606,0.67474697){\makebox(0,0)[lt]{\lineheight{1.25}\smash{\begin{tabular}[t]{l}4\end{tabular}}}}%
    \put(0.83866923,0.83739919){\makebox(0,0)[lt]{\lineheight{1.25}\smash{\begin{tabular}[t]{l}5\end{tabular}}}}%
  \end{picture}%
\endgroup%

		\caption{The virtual knot $K_{5.{2429}}$,}
		\label{fig:5-2429-vknot}
	\end{subfigure}
	\hspace*{\fill}
	\begin{subfigure}[b]{0.3 \textwidth}
		\centering
		\def\svgwidth{\columnwidth}
		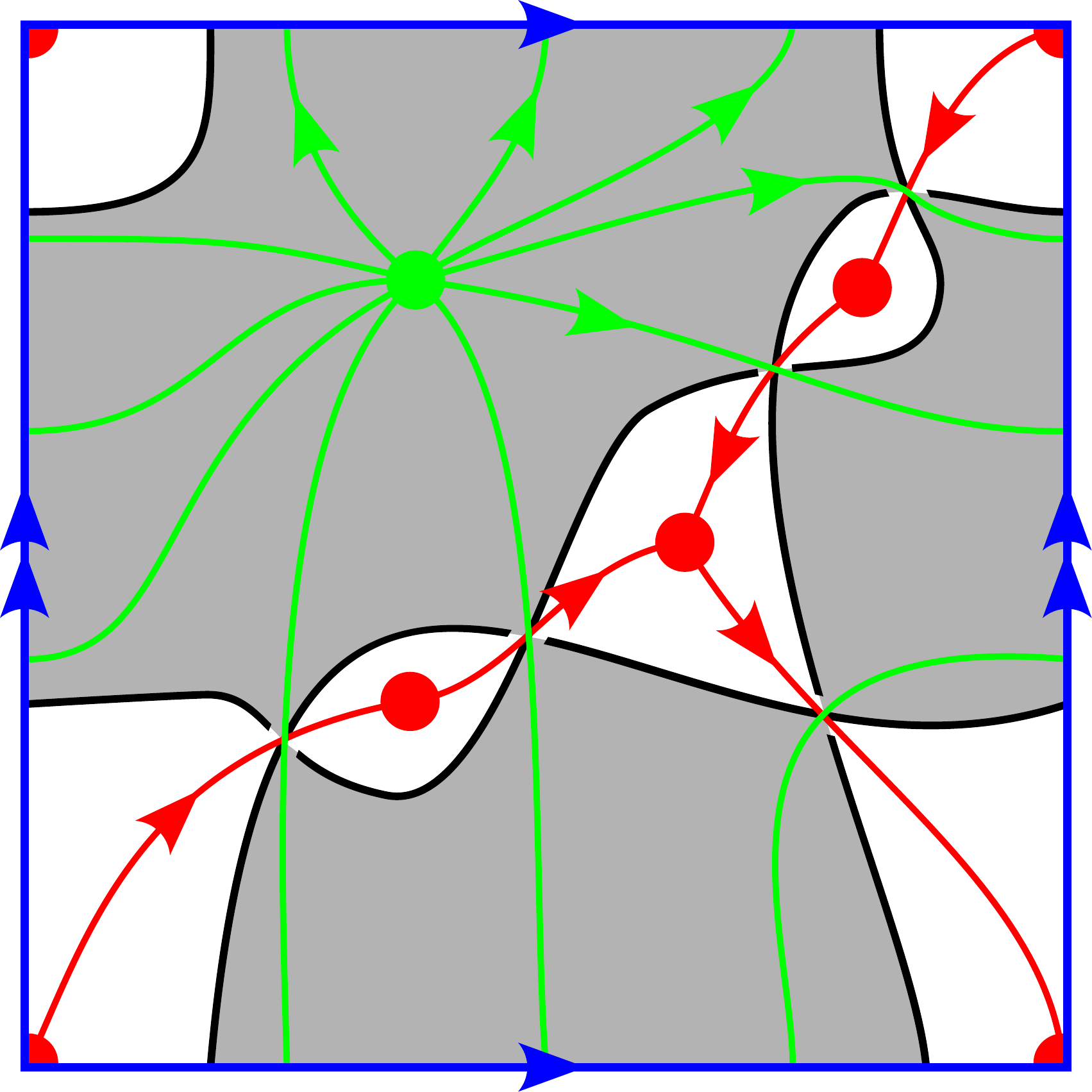
		\caption{\ldots with Tait graphs.}
		\label{fig:5-2429-taits}
	\end{subfigure}
	\hspace*{\fill}
	\caption{The knot $K_{5.2429}$ checkerboard shaded, with black Tait graph shown in green, and white Tait graph shown in red. The edge orientations are explained in Section~\ref{subsec-Flows}.}
	\label{fig:5-2429-example}
\end{figure}
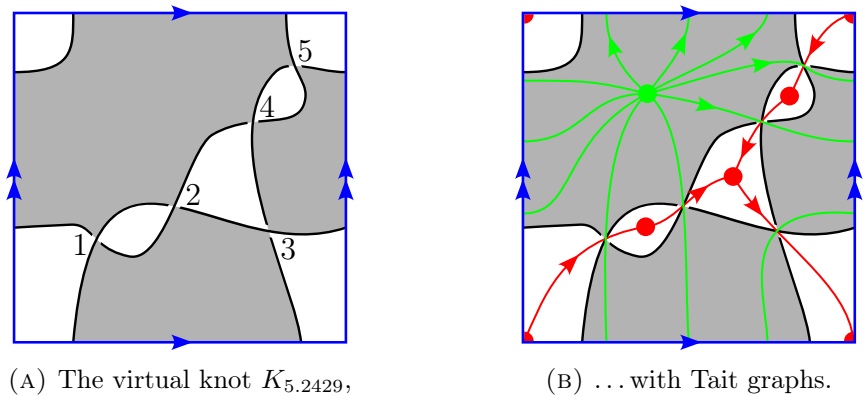

\begin{figure}[hbt]
	\centering
	\def\svgwidth{0.6\columnwidth}
	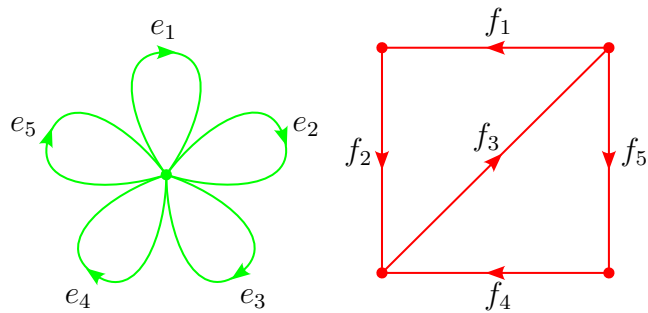
	\caption{The Tait graphs $G_b(D)$ on the left, and $G_w(D)$ on the right, for the diagram in Figure~\ref{fig:5-2429-vknot}, as abstract graphs. } 
	\label{fig:5-2429-taits-abstract}
\end{figure}

\subsection{Mutation}

Conway mutation is an operation on classical links that is generally very difficult to detect. Let $L$ be a link in $S^3$ and $B$ a 3-ball intersecting $L$ in two arcs. Let $S =\partial B$, and assume that $L$ meets $S$ transversely in four points. By cutting along $S$, applying an orientation-preserving involution $\tau$, and regluing, we obtain a new link $L'$, a \textit{mutant} of $L$. If the involution $\tau$ preserves the orientations of the four points $S \cap K$, then $K'$ is said to be a positive mutant.

This operation can be easily visualized on link diagrams. Choose a disc whose boundary intersects the link diagram in four points and perform a $\pi$ rotation of the disc around an axis in the plane of the diagram, or perpendicular to the plane, which preserves the set of four intersection points, see Figure~\ref{fig:kinoshita-terasaka-mutants}. Mutation preserves the checkerboard colouring of a diagram as well as the types (A or B) of the crossings. In particular, the mutant of an alternating diagram is again alternating. 

\begin{figure}
	\centering
	\hspace*{\fill}
	\begin{subfigure}[b]{0.4 \textwidth}
		\centering
		\def\svgwidth{\columnwidth}
		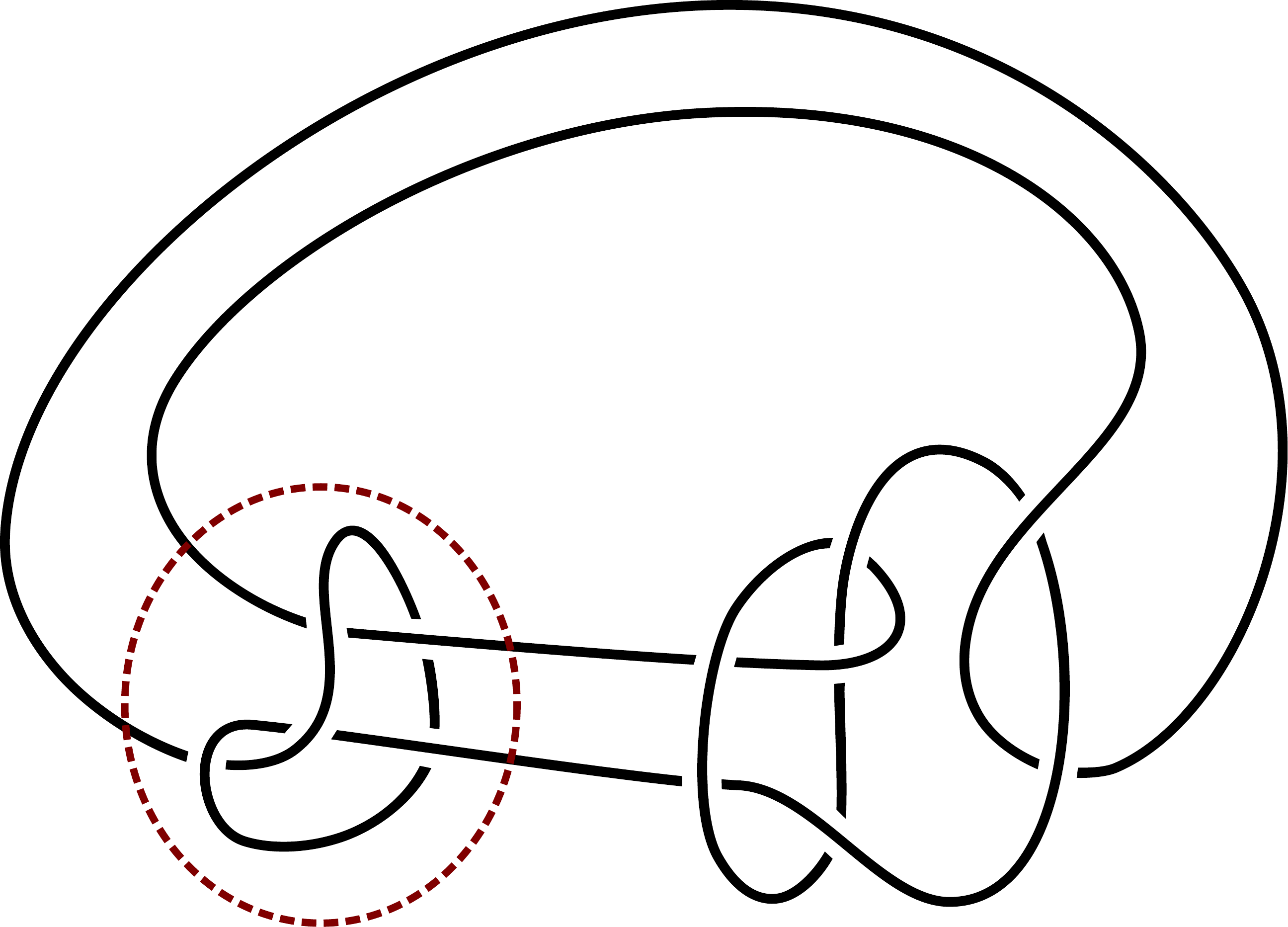
		\caption{The Kinoshita-Terasaka Knot}
		\label{fig:kinoshita-terasaka-knot}
	\end{subfigure}
	\hspace*{\fill} \hspace*{\fill}	\hspace*{\fill}
	\begin{subfigure}[b]{0.4 \textwidth}
		\centering
		\def\svgwidth{\columnwidth}
		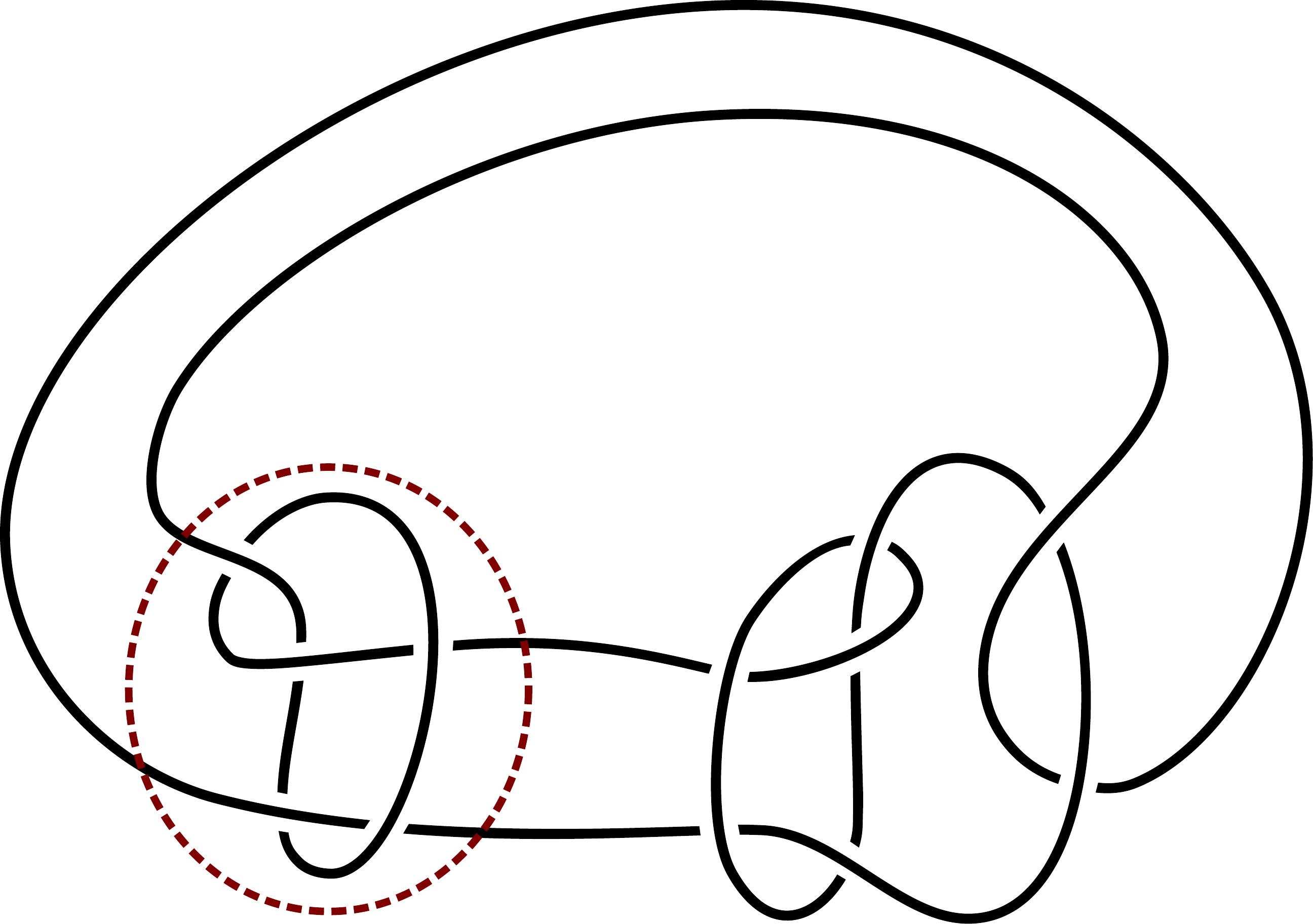
		\caption{The Conway Knot}
		\label{fig:conway-knot}
	\end{subfigure}
	\hspace*{\fill} 
	\caption{The Kinoshita-Terasaka and Conway knots are mutants with trivial Alexander polynomial. 
The KT knot is slice, and the Conway knot is not slice \cite{Piccirillo-2020}. The projections used here are those in \cite[Figure~2.32]{Adams}.}
	\label{fig:kinoshita-terasaka-mutants}
\end{figure}

There are two ways to generalise mutation to links in thickened surfaces: disk mutation and surface mutation. \textit{Disk mutation} is directly analogous to mutation of classical links: choose a disk $B \subseteq \Si$ whose boundary intersects the link diagram in four points, and rotate (flip) it by $\pi$ in the same way. The resulting link is a disk mutant of the original. 

For \textit{surface mutation}, the chosen subsurface still needs to have boundary which intersects the link diagram in four points, but it could have multiple boundary components, or contain handles. This invasive notion of mutation does not necessarily preserve the alternating property: see Example~\ref{example-AnnularMutation}. 
In this paper we always consider disc mutation, unless otherwise stated.

\begin{example}\label{example-AnnularMutation}
\begin{figure}  
\centering
\includegraphics[height=22mm]{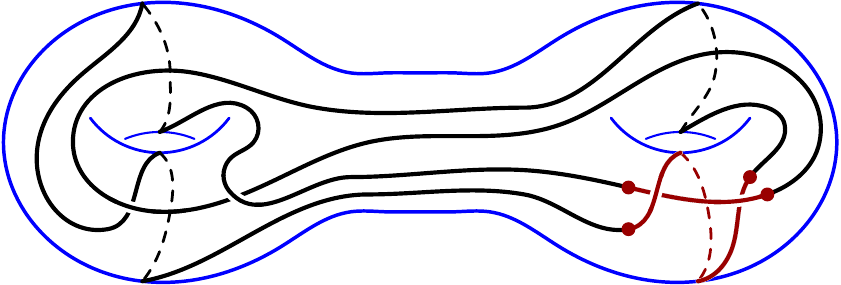} \quad
\includegraphics[height=22mm]{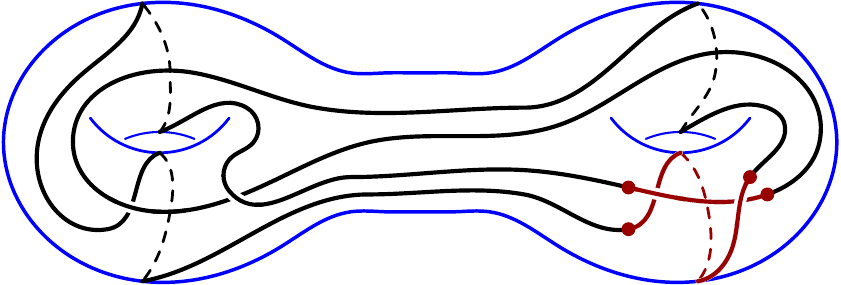} 
\caption{Surface mutation of knots in a genus two surface, with mutant region shown in red. We would like to thank Homayun Karimi for sharing these figures.} \label{fig-mutation}
\end{figure}
Figure~\ref{fig-mutation} is an example of surface mutation of a knot in the closed surface of genus 2. The mutation is carried out by cutting out an annulus, flipping it so as exchanging the two boundary components, and regluing. The part of the knot diagram that is affected is shown in red.  
The knot on the left of Figure~\ref{fig-mutation} is alternating whereas the one on the right is not. Thus, surface mutation does not preserve the property that the diagram is alternating. 

The two knots can be distinguished using knot signatures. The knot on the left has knot signatures $\{2,-2\}$ (computed in terms of the two checkerboard surfaces). The knot on the right has knot signatures $\{0,0\}$. Note that the knot on the left is non-slice (by Theorem 3.9, \cite{Boden-Karimi-2021}), whereas the knot on the right is slice. We leave it as an exercise to find saddle moves to slice this knot. 
Therefore, surface mutation does not preserve the concordance class of the underlying knot.
\hfill $\Diamond$
\end{example}


\section{Lattices}  \label{section-Lattices}

In this section, we generalise some results about the lattice of integer flows to graphs in surfaces. As in the classical case, the lattices of integer flows of the Tait graphs are invariants of alternating links in surfaces. 

The lattice of integer flows -- whose definition is reviewed below in Section~\ref{subsec-Flows} -- is an invariant of abstract graphs, thus it can be applied for surface Tait graphs with no change from the classical setting. This also enables the definition of $d$-invariants for surface links in Section~\ref{subsec-dInv}. 

The key departure from the classical setting is that the black and white Tait graphs are no longer planar duals, but rather surface duals. This makes the relationship between their lattices of integer flows and $d$-invariants more intricate.


\subsection{Integer lattices}\label{subsec-Lattice}
A \textit{lattice} $\La$ is a finitely generated free abelian group with a nondegenerate symmetric bilinear form $\langle \; ,\, \rangle \co \La \times \La \to \QQ$ . The form is  said to be \textit{integral} if $\langle x,y\rangle \in \ZZ$ for all $x,y \in \La$, and it is said to be \textit{even} if $\langle x,y\rangle \in 2\ZZ$ for all $x,y \in \La$. If the form is integral, then the lattice $\La$ is called an \textit{integer lattice}.

Two lattices $\La$ and $\La'$ are said to be \textit{isometric} if there is an isomorphism $\varphi\co \La \to \La'$ of groups such that
$\langle x,y\rangle = \langle \varphi(x),\varphi(y)\rangle'$ for all $x,y \in \La.$

Given a basis $\be=\{b_1,\ldots,b_n\}$ for $\La$,  the \textit{pairing matrix} is denoted $M_{\La,\be}$ and defined to be $(M_{\La,\be})_{i,j}=\langle b_i, b_j \rangle$. If $\be'$ is a different basis for $\La$, then the new pairing matrix is $M_{\La,\be'}=P^\tr M_{\La,\be}P$, where $P$ is the change of basis matrix from $\be$ to $\be'$ and $P^\tr$ is its transpose. Note that $P$ is an integer matrix with determinant 1, so the pairing matrices are unimodular congruent.

Recall that two integer matrices $M$ and $N$ are said to be \textit{unimodular congruent} if there exists an integer matrix $P$ of determinant 1, such that $N=P^\tr MP$. There is a one-to-one correspondence between isometry classes of integer lattices and unimodular congruence classes of nondegenerate symmetric integer matrices.

The dual lattice of $\La$ is the (rational) lattice 
$$\La^* = \{ x \in \La \otimes \QQ \mid \langle x,y\rangle \in \ZZ \text{ for all } y \in \La \}.$$ 
The \textit{discriminant group} of $\La$ is the quotient
$\wbar{\La} = \La^*/ \La.$ This is a finite group of order $\det(M_{\La,\be})$, hence this determinant is also called the \textit{discriminant} of $\La$. For $x \in \La^*$, let
$|x|=\langle x,x \rangle$ denote the norm (squared) of $x$, and $\wbar{x}$ its image in $\wbar{\La}$.

\subsection{Cuts, flows, and surface duality} \label{subsec-Flows}
The lattices of integer cuts and flows are lattice-valued invariants associated to graphs. We begin with a short review of some of the necessary graph theory terminology.

A connected graph $G$ is \textit{two-edge-connected} if it remains connected after removing any one edge. 
A \textit{2-isomorphism} of graphs $G$ and $G'$ is a cycle-preserving bijection of the edge sets $E(G)$ and $E(G')$. The graphs $G$ and $G'$ are called \textit{2-isomorphic} if a 2-isomorphism exists. We note that for two-edge-connected graphs, the notion of 2-isomorphism is equivalent to isomorphism of the associated cycle matroids.

To define the lattices of integer flows and cuts for a graph $G$, fix an arbitrary orientation $\omega_0$ of $G$. This makes $G$ into a one-dimensional CW-complex, with cellular chain complex
$$0 \to C_1(G; \ZZ) \stackrel{\partial}{\lto} C_0(G;\ZZ) \to 0, $$
where $\partial(e)= \operatorname{beg}(e)-\operatorname{end}(e)$, the difference of the beginning and ending vertices of $e$.
View the first chain group $C_1(G;\ZZ) = \ZZ^{E(G)}$ as an integer lattice with Euclidean inner product, by declaring the edges of $G$ to be orthonormal: $\langle e_i, e_j\rangle=\delta_{ij}$. 

The adjoint map $\partial^*\co C_0(G; \ZZ) \to C_1(G; \ZZ)$ assigns to each vertex $v$ the difference of the outgoing and incoming edges incident to $v$:
$$\partial^*(v)= \sum_{e\in \Out(v)} e - \sum_{e'\in \In(v)} e'= \sum_{e\in E(G)} \langle \partial(e),v\rangle \cdot e.$$

\begin{definition} The lattice of integer flows $\cF(G)$ of $G$ is the free abelian group $$\flows(G)=\ker(\partial)=H_1(G;\ZZ),$$
with the inner product inherited from $C_1(G;\ZZ)$.

The lattice of integer cuts $\cC(G)$ of $G$ is the image of $\partial^* \co C_0(G; \ZZ) \to C_1(G; \ZZ)$, 
again with the inner product inherited from $C_1(G;\ZZ)$. 
\end{definition}

Note that while the definitions of $\flows(G)$ and $\cuts(G)$ depend on the choice of orientation $\omega_0$, their isometry types as lattices are independent of this choice. This explains why this choice is not reflected in the notation. 

Given a cycle (closed walk) in $G$, the signed sum of the edges of the walk -- positive where the walk direction agrees with the edge orientation, and negative otherwise -- is an integer flow in $\flows(G)$. Given a partition of the vertex set of $G$, $V_G=V_1 \sqcup V_2$, the signed sum of edges connecting $V_1$ and $V_2$ -- positive when oriented $V_1 \to V_2$, negative otherwise -- is an integer cut in $\cuts(G)$.

It has long been known \cite{Backer-et-al, Watkins-1990} that $\cF(G)$ and $\cC(G)$ are invariants of the 2-isomorphism type of the graph $G$. The converse was proved in \cite{Watkins-1994} using Laplacian matrices, then later by Caporaso and Viviani \cite{Caporaso-Viviani} in the context of tropical curves, and Su and Wagner \cite{Su-Wagner} in the context of regular matroids. The result is widely known as the \textit{Discrete Torelli Theorem}, after an analogous theorem in algebraic geometry:

\begin{theorem}[Discrete Torelli Theorem]\label{thm-DiscTorelli}
The lattice of integer flows $\cF(G)$ of a two-edge-connected graph $G$ uniquely determines the 2-isomorphism type of $G$. The same is true for the lattice of integer cuts $\cC(G)$ of a loopless graph.
\end{theorem}

Next, we turn our attention to lattices of integer flows and cuts associated to Tait graphs. For convenience, we shorten $\cF(G_b(D))$ and $\cC(G_b(D))$ to $\cF_b(D)$ and $\cC_b(D)$, respectively, and similarly for the white graphs. We recall the definition and basic properties of surface graphs, which can be found, for example, in \cite[Chapter 4.1]{Mohar-Thomassen}.

Let $\Si$ be a closed, oriented surface and $G \stackrel{\iota}{\hookrightarrow} \Si$ a cellularly embedded connected graph, that is, $\Si \sm G$ is a disjoint union of discs. Surface embeddings are considered up to orientation preserving homeomorphism of the pair. The data of a surface embedding for $G$ is equivalent to a choice of cyclic orderings for the edges at each vertex of $G$; such data is also known as a \textit{fat graph}, a \textit{ribbon graph}, or a \textit{rotation system} \cite[Chapter 3.2]{Mohar-Thomassen}. 

The \textit{surface dual} of a cellular surface graph $G \stackrel{\iota}{\hookrightarrow} \Si$ is the graph $G^* \stackrel{j}{\hookrightarrow} \Si$ with one vertex for every disc of $\Si \sm G$, and for each edge $e$ of $G$, a dual edge $e^*$ connecting the vertices lying in the discs on either side of $e$. Note that if $e$ borders the same disc on both sides (that is, if $e$ is a \textit{bridge}), then the dual edge $e^*$ is a loop.

\begin{lemma}\label{lem:face-vertex}
There is a unique vertex of $G$ in each connected component of $\Si \sm G^*$.
\end{lemma}

\begin{proof}
Let $v$ denote an arbitrary vertex of $G$. The dual edges to the edges of $G$ incident to $v$ form a cycle in $G^*$ which separates $v$ from all other vertices. 

Conversely, each connected component of $G^*$ has a cycle of edges, called a \textit{facial walk}, as its boundary. By construction, the edges of $G$ dual to the edges of this cycle must terminate within the component, therefore, the component must contain a vertex of $G$.
\end{proof}

\begin{lemma}\label{lem:dualcell}
If $G\hookrightarrow \Si$ is cellularly embedded, then so is $G^*$.
\end{lemma}

\begin{proof}
Let $V_G$, $E_G$, and  $F_G$ denote the sets of vertices, edges, and faces of $G$, respectively. Similarly, let $V_{G^*}$, $E_{G^*}$ denote the sets of vertices and edges of $G^*$, and let $F_{G^*}$ denote the regions of $\Si \sm G^*$. 

Since $G$ is cellularly embedded in $\Si$, the faces in $F_G$ are all discs. It follows that the Euler characteristic of $\Si$ is given by $\chi(\Si)=|V_G|-|E_G|+|F_G|$. 
On the other hand, $\chi(\Si)\geq |V_{G^*}| - |E_{G^*}| + |F_{G^*}|$, with equality if and only if $G^*$ is cellularly embedded. 
By definition, $|E_G|=|E_{G^*}|$ and $|F_G|=|V_{G^*}|$. By Lemma~\ref{lem:face-vertex}, $|V_G|=|F_{G^*}|$. Thus, equality holds, and it follows that $G^*$ is cellularly embedded.
\end{proof}

\begin{corollary}
If $G\hookrightarrow \Si$ is cellularly embedded, then $G^{**}=G$.  
\end{corollary}

The orientation $\omega_0$ of $G$ induces a dual orientation $\omega_0^*$ of $G^*$, using the right-hand rule based on the orientation of $\Si$. Note that $\omega_0^{**}=-\omega_0$. Our next goal is to compare the cut and flow lattices of $G$ and $G^*$, given the orientations $\omega_0$ and $\omega_0^*$. Note that integer flows and cuts are invariants of abstract graphs and do not depend on their embeddings, however, the construction of the surface dual $G^*$ depends on the embedding of $G$.

The inclusion maps of $G$ and $G^*$ into $\Si$ induce canonical maps  $\iota_*\co H_1(G; \ZZ) \to H_1(\Si; \ZZ)$ and $j_* \co H_1(G^*; \ZZ) \to H_1(\Si; \ZZ)$. Let 
\begin{equation} \label{eqn-sublattice}
\flows^\circ(G) =\ker \iota_*  \quad \text{and} \quad \flows^\circ(G^*)=\ker j_*. 
\end{equation}
Then $\flows^\circ(G) \subseteq \flows(G)$ and $\flows^\circ(G^*)\subseteq \flows(G^*)$ are
sublattices of the flow lattices of $G$ and $G^*$, respectively.
These sublattices are invariants of the surface embeddings of $G$ and $G^*$.

If $G$ and $G^*$ are \textit{planar} duals, then it is a classical fact, known as cut-flow duality, that $\cF(G)\cong \cC(G^*)$ as lattices, and vice versa.
The following theorem is a generalisation of this to surface graphs:

\begin{theorem}\label{thm:SurfaceCutFlowDuality}
There is a short exact sequence of abelian groups
\begin{equation}
0 \to \cuts(G^*) \stackrel{\varphi}{\lto} \flows(G) \stackrel{\iota_*}{\lto} H_1(\Si; \ZZ) \to 0,
\end{equation}
where $\varphi(\sum_i c_i e_i^*)= \sum_i c_i e_i$. In particular $\cuts(G^*) \cong \flows^\circ(G)$ as lattices.
\end{theorem}

\begin{proof}
The lattice $\cuts(G^*)$ is generated by the so-called \textit{vertex cuts} of $G^*$: namely, for a vertex $v^* \in V_{G^*}$, the vertex cut 
$$k_{v^*}=\partial^*(v)=\sum_{e^*_r\in \In(v^*)} e^*_r - \sum_{e^*_s\in \Out(v^*)}e^*_s.$$
Furthermore, $\varphi(k_v)$ is the flow in $\flows(G)$ corresponding to the facial walk around $v^*$. Thus, $\varphi(\cuts(G^*))\subseteq \flows(G)$.
Since $G$ is cellularly embedded, each facial walk of $G$ bounds a disc. In particular, $\iota_*(\varphi(k_v^*))=0$, therefore $\operatorname{im} \varphi \subseteq \ker \iota_*$.

On the other hand, let $N(G)$ denote a regular neighbourhood of $\im(\iota)$. Then $N(G)$ is obtained from $\Si$ by removing a finite number of disjoint discs, one in each face of $G$:
$N(G)=\Si \sm (D_1 \sqcup \dots \sqcup D_k)$. Let $\ga_i$ denote the boundary curve $\partial D_i$. The curves $\ga_1,\ldots, \ga_k$ generate $\ker \iota_*$, and they represent exactly the flows in $\flows(G)$ corresponding to the facial walks. Hence, $\im \varphi =\ker \iota_*$, and $\cuts(G^*) \cong \flows^\circ(G)$ is an isometry of lattices since $\iota_*$ is inner product preserving.
\end{proof}

\begin{example} We now return to our running example of $K_{5.2429}$ -- see Example~\ref{ex-Tait} -- and compute the lattices of integer flows and cuts for the Tait graphs shown in Figures~\ref{fig:5-2429-example}~and~\ref{fig:5-2429-taits-abstract}. 

Since in this example every edge of $G_b(D)$ is a loop, $\partial=0$, and $\cF_b(D)=C_1(G_b(D))=\spans\{e_{1}, e_{2}, e_{3}, e_{4}, e_{5}\}$, with the Euclidean inner product given by the identity matrix $I_5$.
It follows that $G_b(D)$ has no non-trivial cuts, so $\cC_b(D)=\{0\}$.

The lattice of integer flows for $G_w(D)$ is generated by the two triangles
\begin{equation}\label{eq:FwBasis}
    \cF_w(D)=\spans\{f_{1} + f_{2} + f_{3}, f_3 + f_{4} + f_{5}\},
\end{equation}
and in this basis the pairing matrix is
\[\begin{bmatrix}
3 & 1 \\
1 & 3
\end{bmatrix}\,.\]
The lattice of integer cuts is of rank three, with one basis given by the vertex cuts corresponding to three of the four vertices (the fourth is the sum of these):
\begin{equation}\label{eq:CwBasis}
    \cC_w(D)=\spans\{- f_{1} + f_{2}, - f_2 + f_3 - f_4, f_4 - f_5 \},
\end{equation}
With corresponding pairing matrix
\[\begin{bmatrix}
\phantom{+}2 & -1 & \phantom{+}0 \\
-1 & \phantom{+}3 & -1 \\
\phantom{+}0 & -1 & \phantom{+}2
\end{bmatrix}\,.\]

It is clear that $\cF_b(D)$ and $\cC_w(D)$ are non-isomorphic as they are of different ranks, and so are $\cF_w(D)$ and $\cC_b(D)$. To verify the statement of Theorem~\ref{thm:SurfaceCutFlowDuality}, we compute the kernel of $\iota_*$ in $\cF_b(D)$. Let $H_1(\Si; \ZZ)= \spans\{\gamma_1, \gamma_2\}$, where $\gamma_1$ is represented by the vertical side of the square in Figure~\ref{fig:5-2429-taits}, and $\gamma_2$ is represented by the horizontal side. Then we can read off Figure~\ref{fig:5-2429-taits} that $\iota_*(e_1)=\gamma_1$, $\iota_*(e_2)=\gamma_1$, $\iota_*(e_3)=\gamma_1+\gamma_2$, $\iota_*(e_4)=\gamma_2$, and $\iota_*(e_5)=\gamma_2$. Thus, 
\begin{equation}\label{eq:F0bBasis}
\cF^\circ_b(D)=\spans\{-e_1+e_2, -e_1+e_3-e_4, e_4-e_5\}.
\end{equation}
Since $\varphi(f_i)=e_i$ for $i=1,\ldots ,5$, it is clear that $\varphi(\cC_w(D))=\cF^\circ_b(D)$, and the pairing matrix for $\cF^\circ_b(D)$ is unimodular congruent to that of $\cC_w(D)$ by virtue of being equal.

On the other hand to compute $\cF^\circ_w(D)$, observe that $j_*$ is injective, as the rank of $\cF_w(D)$ is 2. Indeed, $j_*(f_1+f_2+f_3)=\gamma_2$ and $j_*(f_3+f_4+f_5)=\gamma_1$. Thus, $\cF^\circ_w(D)=\{0\}\cong \cC_b(D)$, as expected.
\hfill $\Diamond$
\end{example}

In Section~\ref{section-GLform} we prove that in fact $\cF_b(D)$, $\cF^\circ_b(D)$ -- and therefore, the corresponding cut lattices as well -- are invariants of alternating links in thickened surfaces, and unchanged up to isometry by disc mutation. In the classical setting Greene~\cite{Greene-2011} shows that these invariants are complete up to Conway mutation. In Section~\ref{subsec:noncomplete} we examine the completeness of these invariants.


\subsection{The \textit{d}-invariant and chromatic duality}\label{subsec-dInv}
The $d$-invariant was originally defined by Ozsv\'ath and Szab\'o \cite{Ozsvath-Szabo-2003a} in terms of Heegaard Floer homology.
In \cite{Greene-2011}, Greene gives a purely lattice-theoretic definition of the $d$-invariant for alternating links. In this section, we adopt this approach to define the $d$-invariant for alternating links in thickened surfaces and virtual alternating links. We establish duality results relating the $d$-invariants of the flow and cut lattices of surface dual graphs.

We start with a review of the necessary lattice theory:

\begin{definition}
A \textit{characteristic covector} for a lattice $\La$ is an element $\xi \in \La^*$ satisfying $\langle \xi,x\rangle = |x| \,\, (\text{mod } 2)$ for all $x\in \La.$ 
\end{definition}
Let
\[
    \Char(\La) = \{ \xi \in \La^* \mid \langle \xi,x\rangle = |x| \quad (\text{mod } 2) \}
\]
denote the set of characteristic covectors for $\La$ and set
$$\cX(\La) = \Char(\La) \quad (\text{mod } 2 \La).$$
Note that $\cX(\La)$ is a torsor over $\wbar{\La} =\La^*/\La$, the discriminant group, viewed as $2\La^*/2\La$. 

For $\xi \in \Char(\La)$, let $[\xi]$ denote its image in $\cX(\La).$
Define the \textit{short characteristic covectors} as
$$ {\Short}(\La) = \{ \chi \in \Char(\La) \mid |\chi| \leq |\chi'| \text{ for all } \chi' \in [\chi] \}.$$

\medskip
For $[\xi] \in \cX(\La)$, define the $\rho$-invariant
$$\rho([\xi]) \equiv \frac{|\xi|-\si(\La)}{4} \quad (\text{mod } 2 ),$$
where $\si(\La)$ denotes the signature of the pairing $\langle \, ,\, \rangle$ on $\La.$
If $\xi \in \Char(\La)$ is another characteristic covector with $[\xi'] =[\xi]$, then $\xi'-\xi= 2 x$ for some $x\in \La$ and 
$$|\xi'|=|\xi+2x|= |\xi| + 4 (\langle \xi, x\rangle +|x|) \equiv |\xi| \quad (\text{mod } 8).$$
This shows that $\rho([\xi])$ is independent of the choice of representative for $[\xi]$, and it gives rise to a well-defined map 
\begin{equation}\label{eqn-rho}
\rho\co \cX(\La)\lto  \QQ/2\ZZ.
\end{equation}

If  $\langle \; ,\, \rangle$ is positive definite, then the map $\rho$ in \eqref{eqn-rho} lifts to a well-defined $\QQ$-valued map, the $d$-invariant:

\begin{definition}\label{def-dInv}
For a positive definite lattice $\La$ of rank $n$, the $d$-invariant $(\cX(\La), d)$ is given by the torsor $\cX(\La)$ along with the map
$$d \colon \cX(\La) \lto \QQ,$$ $$d([x])= \min \left\{ \frac{\langle x', x'  \rangle -n}{4} \; \middle| \; x' \in [x]  \right\}.$$ 
\end{definition}

\begin{definition}\label{def-dInvIso}
Given two positive definite lattices $\La, \La'$ with $d$-invariants  
\[
d \colon \cX(\La) \lto \QQ \quad \text{and} \quad d' \colon \cX(\La') \lto \QQ,
\]
an \textit{isomorphism} of $d$-invariants $(\cX(\La),d) \cong (\cX(\La'),d')$ consists of a group isomorphism $\psi \co \wbar{\La} \to  \wbar{\La'}$
and $\psi$-equivariant bijection $\varphi \co \cX(\La) \to \cX(\La')$ such that $d'= d \circ \varphi.$ 
\end{definition}

For a graph $G$, the lattices $\cF(G)$ and $\cC(G)$ are \textit{primitive complementary} sublattices (mutual orthogonal complements) in $C_1(G;\ZZ)\cong\ZZ^{E(G)}$, and the short vectors of $\ZZ^{E(G)}$ surject to the short vectors of $\cF(G)$ and $\cC(G)$ by orthogonal projection \cite[Corollary 3.4]{Greene-2011}. This is a useful fact for computing the $d$-invariants of $\cF(G)$ and $\cC(G)$, as $\Short(\ZZ^n)$ is simply the set of $(\pm 1)$-vectors. It also follows \cite[Corollary 3.4]{Greene-2011} that
for any graph $G$, there is a canonical isomorphism of $d$-invariants
\[
    (\cX(\cF(G)), d_{\cF(G)})\xrightarrow{\cong}(\cX(\cC(G)), -d_{\cC(G)}).
\]

If $G$ is a planar graph embedded in $S^2$, and $G^*$ its planar dual, then we also have, by cut-flow duality, that $\cF(G)\cong \cC(G^*)$ and  $\cF(G^*)\cong \cC(G)$. Thus, all $d$-invariants of cuts and flows of $G$ and $G^*$ are in the same isomorphism class up to sign:
$$(\cX(\cC(G^*)), d_{\cC(G^*)}) \cong (\cX(\cF(G)), d_{\cF(G)}) \cong  (\cX(\cC(G)), -d_{\cC(G)}) \cong (\cX(\cF(G^*)), -d_{\cF(G^*)}).$$

The following result establishes the analogous statement for surface graphs. We call this \textit{chromatic duality}: 
\begin{theorem}[Chromatic duality]\label{thm-ChromDuality}
For a cellular surface graph $G\hookrightarrow \Si$ with surface dual $G^*$, the associated $d$-invariants belong to two equivalence classes, as follows:
$$(\cX(\cF^\circ(G)), -d_{\cF^\circ(G)}) \cong (\cX(\cC(G^*)), -d_{\cC(G^*)}) \cong  (\cX(\cF(G^*)), d_{\cF(G^*)}),$$
$$(\cX(\cF^\circ(G^*)), -d_{\cF^\circ(G^*)}) \cong (\cX(\cC(G)), -d_{\cC(G)}) \cong  (\cX(\cF(G)), d_{\cF(G)}).$$ 
\end{theorem}
\begin{proof}
    Regarding $G$ and $G^*$ as abstract graphs, it remains true that 
    \[
    (\cX(\cF(G)), d_{\cF(G)})\cong (\cX(\cC(G)), -d_{\cC(G)}), \, (\cX(\cF(G^*)), d_{\cF(G^*)})\cong (\cX(\cC(G^*)), -d_{\cC(G^*)}).
    \]
    However, it is no longer true that $\cF(G)\cong \cC(G^*)$. Instead,  Theorem~\ref{thm:SurfaceCutFlowDuality} implies that for the kernels of $\iota_*$ and $j_*$, there are isometries  $\cF^\circ(G)\cong \cC(G^*)$ and $\cF^\circ(G^*)\cong \cC(G)$. \qedhere
\end{proof}

\begin{remark}
Chromatic duality is a general principle that relates invariants of links in thickened surfaces defined in terms of one checkerboard surface with those of its chromatic dual under restriction. It asserts that the link invariants for the black surface $F_b$ coincide with the invariants for the white surface $F_w$ on the kernel of $H_1(F_w;\ZZ) \to H_1(\Si \times I;\ZZ)$, and vice versa.

One instance of chromatic duality is Theorem 5.4 of \cite{Boden-Chrisman-Karimi} relating the signature, determinant, and nullity of the Gordon-Litherland pairing $\cG_{F_b}$ for the black surface with the corresponding invariants of the Goeritz matrices for its chromatic dual. Another instance is Theorem 4.9 of \cite{Boden-Karimi-2021} relating the Brown invariants of the Gordon-Litherland pairing $\cG_{F_b}$ for the black surface with the corresponding invariants of the Goeritz matrices of its chromatic dual. We review the definition of the Gordon-Litherland pairing in Section~\ref{section-GLform}.
\end{remark}

\begin{example}
For our running example of $K_{5.2429}$, we compute the $d$-invariants associated to the lattices $\cF^\circ_{b/w}$, $\cC_{b/w}$ and $\cF_{b/w}$ of the surface dual graphs $G_b=G_{b}(D)$ and $G_w=G_{w}(D)$.

\begin{table}[H]
\centering
{ \def\arraystretch{1.2} \footnotesize
\begin{tabular}{|c|c|c|c|c|} \hline

${\boldsymbol{\Short}(\ZZ^5)}$ & $\boldsymbol{x}$ & $ \boldsymbol{x_{E}}$ & $\boldsymbol{\langle x, x\rangle}$ & $\boldsymbol{d([x])}$\\ \hline \hline
$(1, 1, 1, 1, 1)$&       $(\tfrac{1}{4}, \tfrac{1}{4}, \tfrac{1}{2})$ &    $(\tfrac{1}{4}, \tfrac{1}{4}, -\tfrac{1}{4}, -\tfrac{1}{4}, \tfrac{1}{2})$&   $\tfrac{1}{2}$&      $-\tfrac{5}{8}$\\ \hline
$(1, 1, 1, 1, -1)$&      $(-\tfrac{1}{4}, -\tfrac{1}{4}, -\tfrac{1}{2})$&  $(-\tfrac{1}{4}, -\tfrac{1}{4}, \tfrac{1}{4}, \tfrac{1}{4}, -\tfrac{1}{2})$&  $\tfrac{1}{2}$&      $-\tfrac{5}{8}$ \\ \hline
$(1, 1, 1, -1, 1)$&      $(\tfrac{1}{2}, \tfrac{3}{2}, 1)$  &     $(\tfrac{1}{2}, \tfrac{1}{2}, \tfrac{1}{2}, -\tfrac{3}{2}, 1)$&      $4$&        $\tfrac{1}{4}$ \\ \hline
$(1, 1, 1, -1, -1)$&     $(0, 1, 0)$&           $(0, 0, 1, -1, 0)$&              $2$&        $-\tfrac{1}{4}$ \\ \hline
$(1, 1, -1, -1, 1)$&     $(\tfrac{3}{4}, \tfrac{3}{4}, \tfrac{3}{2})$&     $(\tfrac{3}{4}, \tfrac{3}{4}, -\tfrac{3}{4}, -\tfrac{3}{4}, \tfrac{3}{2})$&   $\tfrac{9}{2}$&      $\tfrac{3}{8}$\\ \hline
$(1, -1, 1, 1, 1)$&      $(1, 0, 0)$&           $(1, -1, 0, 0, 0)$&              $2$&        $-\tfrac{1}{4}$\\ \hline
$(1, -1, 1, -1, 1)$&     $(\tfrac{5}{4}, \tfrac{5}{4}, \tfrac{1}{2})$&     $(\tfrac{5}{4}, -\tfrac{3}{4}, \tfrac{3}{4}, -\tfrac{5}{4}, \tfrac{1}{2})$&   $\tfrac{9}{2}$&      $\tfrac{3}{8}$\\ \hline
$(1, -1, -1, -1, 1)$&    $(\tfrac{3}{2}, \tfrac{1}{2}, 1)$&       $(\tfrac{3}{2}, -\tfrac{1}{2}, -\tfrac{1}{2}, -\tfrac{1}{2}, 1)$&    $4$&        $\tfrac{1}{4}$\\ \hline
\end{tabular}
}
\caption{\small The $d$-invariant of $\cF^{\circ}_b$. The third column expresses the vector $x$ in the edge basis. }\label{table:F0Gb}
\end{table}

\begin{multicols}{2}

\begin{table}[H]
\centering
{ \def\arraystretch{1.3} \footnotesize
\begin{tabular}{|c|c|}
\hline
$\boldsymbol{x}$ & $\boldsymbol{d([x])}$\\ \hline \hline
$(0, \tfrac{1}{4}, -\tfrac{1}{4})$ & $-\frac{5}{8}$ \\ \hline
$(1, 1, 1)$ &  $-\frac{1}{4}$\\ \hline
$(0, -\tfrac{1}{4}, \tfrac{1}{4})$ &  $-\frac{5}{8}$ \\ \hline
$(0, \tfrac{3}{4}, -\tfrac{3}{4})$ &  $\frac{3}{8}$\\ \hline
$(1, \tfrac{3}{2}, \tfrac{1}{2})$  & $\tfrac{1}{4}$\\ \hline
$(1, 0, 0)$ &   $-\tfrac{1}{4}$\\ \hline
$(1, \tfrac{1}{2}, -\tfrac{1}{2})$ & $\tfrac{1}{4}$\\ \hline
$(2, \tfrac{5}{4}, \tfrac{3}{4})$  & $\tfrac{3}{8}$\\ \hline
\end{tabular}
}

\captionsetup{width=\columnwidth}
\caption{The $d$-invariant of $\cC_w$.}\label{table:CG*b}
\end{table}

\begin{table}[H]
\centering
{ \def\arraystretch{1.3} \footnotesize
\begin{tabular}{|c|c|}
\hline
$\boldsymbol{x}$ & $\boldsymbol{d([x])}$\\ \hline \hline
$(\frac{3}{4}, \frac{3}{4})$& $\frac{5}{8}$\\ \hline
$(1, 0)$&     $\frac{1}{4}$\\ \hline
$(\frac{5}{4}, \frac{3}{4})$&  $\frac{5}{8}$\\ \hline
$(\frac{1}{4}, \frac{1}{4})$&   $-\frac{3}{8}$\\ \hline
$(\frac{1}{2}, -\frac{1}{2})$&   $-\frac{1}{4}$\\ \hline
$(0, 1)$&     $ \frac{1}{4}$\\ \hline
$(-\frac{1}{2}, \frac{1}{2})$&  $-\frac{1}{4}$\\ \hline
$(-\frac{1}{4}, -\frac{1}{4})$& $-\frac{3}{8}$\\ \hline
\end{tabular}
}

\captionsetup{width=\columnwidth}
\caption{The $d$-invariant of $\cF_w$.}\label{table:FG*b}
\end{table}

\end{multicols}

By \cite[Proposition 2.7]{Greene-2011}, the short vectors of $\ZZ^n\cong\ZZ^{E(G_b)}=\ZZ^{E(G_w)}$ -- which are exactly the $\pm 1$-vectors -- surject by orthogonal projection to the short vectors of each of $\cX(\cF_w)$ and $\cX(\cC_w)$. By Theorem~\ref{thm:SurfaceCutFlowDuality}, the $\pm 1$-vectors in $\ZZ^{n}$ also surject to the short vectors of $\cX(\cF^{\circ}_{b})$. Thus, for a short vector $s\in \ZZ^{n}$, and image $x=\pi(e)$ in one of the three characteristic cosets, we have $d([x])=\frac{1}{4}(\langle x, x\rangle - n)$, where $\langle x, x \rangle$ is taken in $\ZZ^{n}$. This calculation is presented in full for $\cF^{\circ}_{b}$ in Table \ref{table:F0Gb}. In each table the $x$ column expresses the short vectors in the bases \eqref{eq:F0bBasis}, \eqref{eq:CwBasis}, and \eqref{eq:FwBasis}, respectively. In all three tables, where multiple short vectors in $\ZZ^{5}$ map to the same short vector modulo $2\La$, only one representative is shown.

It is clear by inspection that the $d$-invariant values in the first and second tables are equal, and in the third table are negatives of the second table. The discriminant group in all cases is $\ZZ_8$, and it is an instructive exercise to verify that the identifications of $\cX$ are compatible with the action of $\ZZ_8$.

The $x$ column in Table \ref{table:CG*b} is written in the basis \eqref{eq:CwBasis}.
The $x$ column in Table \ref{table:FG*b} is with respect to the basis \eqref{eq:FwBasis}.
By \cite[Chapter 2.4]{Greene-2011}, the $d$-invariant of $\ZZ^{n}$ is zero. The lattices $\cF^\circ_{w}$, $\cC_{b}$ and $\cF_{b}$ are Euclidean, so their $d$-invariants are all zero.
\hfill $\Diamond$
\end{example}

\section{The Gordon-Litherland form} \label{section-GLform}
In \cite{GL-1978}, Gordon and Litherland defined a symmetric bilinear form for classical links which simultaneously generalised the forms of Goeritz and Trotter. They used it to give a simple method to compute the link signature.

In \cite{Boden-Chrisman-Karimi}, the Gordon-Litherland pairing was extended to links in thickened surfaces, and it was used to define link signatures in this setting. In \cite{Boden-Karimi-2023b}, the  Gordon-Litherland linking form was introduced, and it was used to define new invariants, namely mock Alexander polynomials and mock Levine-Tristram signatures.

In this section, we recall the definitions of the Gordon-Litherland linking form and pairing, as well as link signatures and mock Alexander polynomials. Our main goal is to relate the  Gordon-Litherland pairing to the lattice of integer flows for alternating links in thickened surfaces.

\subsection{Spanning surfaces}\label{subsec-SpanningSurface}

A \textit{spanning surface} of a link $L \subseteq \Si \times I$ is a compact surface $F \subseteq \Si \times I$ with boundary $\partial F = L$. The surface $F$ may or may not be orientable, but we regard it as unoriented. The definition of the Gordon-Litherland linking form in Section \ref{subsec-GLform} below will be based on a choice of spanning surface for the link.

All classical links have spanning surfaces -- in fact, orientable Seifert surfaces -- but the same is not true for links in thickened surfaces.
Recall from Section \ref{subsec-CCandTaitGraph} that if a link in $\Si \times I$ is checkerboard colourable, then both of the checkerboard surfaces are spanning surfaces for the link. Alternatively, if $L \subseteq \Si \times I$ is a link with spanning surface $F\subseteq \Si \times I$, then the surface $F$ is a finite union of discs and bands. One can apply an isotopy to shrink the disks so their images under projection $\Si \times I \to \Si$ are disjoint from one another and also disjoint from the interiors of the bands. The isotopy of $F$ produces  a link diagram which is checkerboard colourable. In summary, a link in $\Si \times I$ has a spanning surface if and only if it is checkerboard colourable.

Alternatively, Proposition 1.1 of \cite{Boden-Chrisman-Karimi} applies to show that a link $L \subseteq \Si \times I$ admits a spanning surface if and only if its homology class $[L]$ vanishes.

\begin{definition} \label{defn-ssequivalent}
Given links $L, L' \subseteq \Si \times I$ with spanning surfaces $F, F' \subseteq \Si \times I$, we say $F$ and $F'$ are \textit{equivalent} if there exists a homeomorphism $\varphi$ of $(\Si\times I, \Si \times \{0\})$ such that $\varphi(F)=F'$.
\end{definition}

If $F$ and $F'$ are equivalent spanning surfaces, then the links $L$ and $L'$ are equivalent as defined in Section \ref{subsec:Links}.


\subsection{Linking numbers}\label{subsec-LinkingNumbers}
A crucial ingredient in the Gordon-Litherland linking form is the notion of linking numbers. Let $J$ and $K$ be two oriented disjoint simple closed curves in the interior of $\Si \times I$. Consider the relative homology group \linebreak $H_1(\Si \times I \sm J, \Si \times \{1\};\ZZ)$. This group is isomorphic to $\ZZ$, and generated by a meridian $\mu$ of $J$. The \textit{linking number} of $J$ with $K$ is denoted $\lk(J, K)$ and defined as the unique integer $n$ such that 
$$[K] = n\mu \quad \text{ in } \quad H_1(\Si \times I \sm J, \Si \times \{1\};\ZZ).$$ 

Equivalently, $\lk(J, K)$ is the algebraic number of instances that $J$ passes above $K$, where \textit{above} is defined with respect to $t$, the parameter of the interval $I$. A positive (resp. negative) contribution is made to $\lk(J, K)$ when $J$ passes over $K$ such that the coordinate system given by the unit tangent vector to $J$, the unit tangent vector to $K$, and the positive $t$ direction has a positive (resp. negative) orientation, as shown in Table~\ref{tab:linking-number-crossings}.

\begin{table}[ht]
\centering
\footnotesize{
    \begin{tabular}{| c | c | c @{\hspace{10pt}}||@{\hspace{10pt}} c | c | c|}  \hline
    \textbf{Crossing}	& $\boldsymbol{\lk(J, K)}$ & $\boldsymbol{\lk(K, J)}$ & \textbf{Crossing}	& $\boldsymbol{\lk(J, K)}$ & $\boldsymbol{\lk(K, J)}$ \\ \hline 	\hline	 
$\begin{tikzpicture}
\draw[line width=1.50pt, blue,->] (.05,.05) -- (.32,.30) ;
\draw[line width=1.50pt,blue](-.32,-.30) -- (-.05,-.05);
\draw[line width=1.50pt,red,->](-.32,.30) -- (.32,-.30);
\end{tikzpicture}$
& $+1$ & 0 & 
$\begin{tikzpicture}
\draw[line width=1.50pt,white] (-.32,.40) -- (.32,.40);
\draw[line width=1.50pt,red] (-.32,.30) -- (-.05,.05);
\draw[line width=1.50pt, red,->] (.05,-.05) -- (.32,-.30);
\draw[line width=1.50pt,blue,->](-.32,-.30) -- (.32,.30);
\end{tikzpicture}$
& 0 & $-1$ \\ \hline 		
$\begin{tikzpicture}
\draw[line width=1.50pt,blue] (-.32,.30) -- (-.05,.05);
\draw[line width=1.50pt, blue,->] (.05,-.05) -- (.32,-.30);
\draw[line width=1.50pt,red,->](-.32,-.30) -- (.32,.30);
\end{tikzpicture}$
 & $-1$ & 0 & 
$\begin{tikzpicture}
\draw[line width=1.50pt,white] (-.32,.40) -- (.32,.40);
\draw[line width=1.50pt, red,->] (.05,.05) -- (.32,.30) ;
\draw[line width=1.50pt,red](-.32,-.30) -- (-.05,-.05);
\draw[line width=1.50pt,blue,->](-.32,.30) -- (.32,-.30);
\end{tikzpicture}$
& 0 & $+1$ \\ \hline
    \end{tabular}
}
	\caption{Crossings between $J$ (red) and $K$ (blue) and their contribution to $\lk(J, K)$ and $ \lk(K, J)$.}
	\label{tab:linking-number-crossings}
\end{table}

Note that $\lk(J, K)$ is independent of the number of times that $J$ passes under $K$: this is instead captured by $\lk(K, J)$. In general, $\lk(J,K) \neq \lk(K,J)$. The exception is the classical setting of links in $\RR^{2} \times I$ -- or equivalently, in $S^2 \times I$ -- where linking numbers are symmetric as a consequence of the Jordan curve theorem.

A related notion is the \textit{intersection pairing} on the first homology $H_1(\Si;\ZZ)$, denoted $\al\cdot \be$ for homology classes $\al$ and $\be$. It is defined to be the algebraic intersection number of $\al$ and $\be$, where the sign is the sign of the coordinate system given by the derivatives $\hat{a}, \hat{b}$, where $a$ and $b$ are two $1$-chains. 

Let $p\co  \Si \times I \lto \Si$ be the projection map and $p_{*} \co H_1(\Si \times I) \to H_1(\Si)$ the induced map on the first homology groups. The linking numbers in $\Si\times I$ and algebraic intersection numbers on $\Si$ are related:
\begin{equation}\label{eqn-CTlinking}
\lk(J, K) - \lk(K, J) = p_{*}(J) \cdot p_{*}(K)\,.
\end{equation}
See \cite[\S 1.2]{Cimasoni-Turaev} for a proof, or verify it directly by a case analysis using Table~\ref{tab:linking-number-crossings}.


\subsection{The Gordon-Litherland linking form and pairing} \label{subsec-GLform}

Let $L\subseteq \Si\times I$ be a link and $F \subseteq \Si \times I$ a spanning surface for $L$.
The normal bundle $N(F)$ has boundary $\wt{F}$, which is a double cover of $F$. Define the \textit{transfer map}  $\tau\co H_1(F;\ZZ) \lto H_1(\wt{F};\ZZ)$ by $\tau([a]) = [\pi^{-1}(a)]$.
If $[a]$ is represented by a simple closed curve $\al$ on $F$, then $\tau([a])$ is represented by the curve $\tau \al$ obtained by taking the normal pushoffs of $\al$ in both directions. Note that $\tau \al$ will consist of one or two simple closed curves, depending on whether $\pi^{-1}(\al) \subseteq \wt{F}$ is connected or not. 

The following is Definition 3.1 of \cite{Boden-Karimi-2023} extended to links in thickened surfaces.

\begin{definition}\label{def-GLForm}
 Let $L$ be a link in $\Si \times I$, and $F$ a spanning surface for $L$. The \textit{Gordon-Litherland linking form}, or \textit{linking form} for short, is the (not necessarily symmetric)  bilinear form
\[\cL_{F}\co H_1(F;\ZZ) \times H_1(F;\ZZ) \lto \ZZ,\]
\[\cL_{F}(\al, \be) = \lk(\tau\al, \be).\]
\end{definition}

If $F,F'$ are equivalent spanning surfaces in the sense of Definition \ref{defn-ssequivalent}, then their linking forms $\cL_F$ and $\cL_{F'}$ are isomorphic. 

Given a basis $\{\al_{1}, \dots, \al_{n}\}$ for $H_1(F;\ZZ)$, the $n \times n$ matrix with $(M)_{ij} = \lk(\tau \al_{i}, \al_{j})$ is a representative for $\cL_F$. Any such matrix  is called a \textit{mock Seifert matrix} for $\cL_F.$ The word \textit{mock} here refers to the fact that the spanning surface $F$ need not be orientable. Just as with lattices (cf.~Section~\ref{subsec-Lattice}), the mock Seifert matrix changes by unimodular congruence under a change of basis for $H_1(F;\ZZ)$.  

One link invariant that can be derived from mock Seifert matrices is the link determinant. It is defined by setting \begin{equation} \label{eqn-linkdet}
    \det(L) = |\det(A)|,
\end{equation} 
where $F$ is a spanning surface for $L$ and $A$ is a mock Seifert matrix representing the linking form $\cL_F.$ By \cite[Corollary 3.10]{Boden-Karimi-2023}, the link determinant is independent of the choice of spanning surface $F$ for $L$.

Another link invariant is the mock Alexander polynomial. It is defined for knots in thickened surfaces in \cite[Definition 4.4]{Boden-Karimi-2023}, and the same definition works for links.
\begin{definition}\label{def-mock-Alex}
Let  $L$ be a link in $\Si \times I$ with spanning surface $F$. The \textit{mock Alexander polynomial} is given by $\De_{K,F}(t) =\det(t A - A^\tr)$, where $A$ is a mock Seifert matrix for the linking form $\cL_{F}$. Then $\De_{K,F}(t)$ is an element in the ring $\ZZ[t,t^{-1},(1-t)^{-1}]$ and is well-defined up to units.
\end{definition}

For $\De(t), \De'(t) \in \ZZ[t,t^{-1}, (t-1)^{-1}]$, we write $\De(t) \doteq \De'(t)$ if they are equal up to factors $\pm t^k$ and $\pm(t-1)^\ell$. It depends on the spanning surface $F$ up to $S^*$-equivalence. For classical links, Gordon and Litherland showed in \cite[Theorem 11]{GL-1978} that any two spanning surfaces are $S^*$-equivalent. (For an elementary proof, see \cite{Yasuhara}.) Proposition 1.6 of \cite{Boden-Chrisman-Karimi} gives a generalization for non-split checkerboard colourable links in thickened surfaces. It asserts that there are exactly two  $S^*$-equivalence classes of spanning surfaces; one is represented by the black checkerboard surface $F_b$ and the other is represented by the white surface $F_w$.

\begin{example}
Continuing the example of $K_{5.2429}$, we calculate the linking forms $\cL_{F_b}$ and $\cL_{F_w}$ corresponding to the black and white checkerboard surfaces, as in Figure~\ref{fig:5-2429-taits}.
For the black checkerboard surfaces with basis $\{e_1,e_2,e_3,e_4,e_5\}$, the linking form $\cL_{F_b}$ is represented by the mock Seifert matrix
$$ A_b=\begin{bmatrix} 1 & 0 & -1 & -1 & -1 \\ 
0 & 1 & -1 & -1 & -1 \\ 
1 & 1 & \phantom{+}1 & -1 & -1 \\ 
1 & 1 & \phantom{+}1 & \phantom{+}1 & \phantom{+}0 \\ 
1 & 1 & \phantom{+}1 & \phantom{+}0 & \phantom{+}1
\end{bmatrix}. $$
This matrix has determinant 9, thus by Equation \eqref{eqn-linkdet}, this knot has determinant $\det(K_{5.2429})=9$. By Definition \ref{def-mock-Alex} it has mock Alexander polynomial  
$$\De_{K,F_b}(t)=9t^5 - 13t^4 - 6t^3 + 6t^2 + 13t - 9 \doteq 9t^2 + 14t + 9.$$

For the white checkerboard surfaces with basis $\{f_1+f_2 +f_3,f_2+f_4+f_5\}$, the linking form $\cL_{F_w}$ is represented by the mock Seifert matrix
$$ A_w = \begin{bmatrix} -3 & \phantom{+}0 \\ 
\phantom{+}2 & -3 
\end{bmatrix}.$$
This matrix also has determinant 9, and  the mock Alexander polynomial is $\De_{K,F_w}(t)=9t^2-14t+9.$
\hfill $\Diamond$
\end{example}

In practice, it is simpler to think of the Gordon-Litherland form as a sum of its symmetric and anti-symmetric parts. In particular, we give the symmetrisation of $\cL$ its own name: 

\begin{definition}\label{def-GLPairing}  
 Let $L$ be a link in $\Si \times I$, and $F$ a spanning surface for $L$. 
 The \textit{Gordon-Litherland pairing} is the symmetric bilinear form
\[\cG_{F}\co H_1(F;\ZZ) \times H_1(F;\ZZ) \lto \ZZ,\]
\[\cG_{F}(\al, \be) = \tfrac{1}{2}(\lk(\tau \al, \be) + \lk(\tau\be, \al))\,.\]
We say that  the surface $F$ is \textit{positive definite} if the form $\cG_{F}$ is positive definite,
and \textit{negative definite} if $\cG_{F}$ is negative definite.
\end{definition}

If $A$ is a mock Seifert matrix for $\cL_{F}$ with respect to some basis of $H_1(F;\ZZ)$, then its symmetrisation $\frac{1}{2}(A + A^{\tr})$ represents $\cG_{F}$ with respect to the same basis. 

We take a moment to recall the definition of the link signature for the checkerboard surfaces $F_b,F_w$ in terms of the Gordon-Litherland pairing 
from \cite[\S 2]{Boden-Chrisman-Karimi}. 
Let $L \subseteq \Si \times I$ be a link represented by a connected, checkerboard coloured, cellular diagram $D$. Then the checkerboard surfaces $F_b, F_w$ are spanning surfaces for $L$.
For any crossing $x$ of $D,$ referring to Figure \ref{fig:crossing-type}, define the incidence number $$\eta_x = \begin{cases} 1, & \text{ if $x$ is type A,}\\-1, & \text{ if $x$ is type B.} \end{cases}$$
The black and white correction terms are given by (cf.~ \cite[Lemma 2.4]{Boden-Chrisman-Karimi}) 
$$\mu(F_b) = \sum_{x \text{ type II}} \eta_x \quad \text{ and } \quad \mu(F_w) = -\sum_{x \text{ type I}} \eta_x.$$
The link signatures of $L$  with respect to the checkerboard surfaces are defined as:
\begin{equation} \label{eqn-linksignature}
\si(L,F_b) = \sig(\cG_{F_b}) - \mu(F_b) \quad \text{ and } \quad 
 \si(L,F_w) = \sig(\cG_{F_w}) - \mu(F_w).
 \end{equation}
In general,  the two signatures $\si(L,F_b)$ and $\si(L,F_w)$ need not be equal. In fact, according to \cite{Boden-Chrisman-Karimi}, the link signature depends on the choice of spanning surface up to $S^*$-equivalence. As previously noted, there are exactly two  $S^*$-equivalence classes of spanning surfaces, and each is represented by one of the checkerboard surfaces.

For classical links, \cite[Section 1.1]{Greene-2011} Greene observed that the Gordon-Litherland pairing on the black  (positive definite) checkerboard surface of an alternating link diagram coincides with the flow pairing on the lattice of integer flows of the black Tait graph. The Gordon-Litherland pairing on the white (negative definite) surface coincides with the negative of the pairing of the flow lattice of the white Tait graph.
The next theorem presents the analogous results for alternating links in $\Si\times I$:

\begin{theorem} \label{thm-equiv}
Let $D$ be a reduced, alternating, cellular diagram for a link $L\subseteq \Si \times I$ with checkerboard colouring so that all crossings have type A. For the black surface $F_b(D)$, the Gordon-Litherland pairing is isometric to the lattice of integer flows on the black Tait graph $G_b$, that is, $\cG_{F_b(D)} \cong \cF_b(D)$. For the white surface, they are isometric with a change of sign: $\cG_{F_w(D)} \cong -\cF_w(D)$.
\end{theorem}

As a first step, we have the following general framework for computing the Gordon-Litherland pairing.  
Let $D \subseteq \Si$ be a checkerboard coloured link diagram, and let $F \subseteq \Si$ be one of the checkerboard surfaces. We can choose a basis for $H_1(F;\ZZ)$ represented by simple closed curves $\{\al_1,\ldots ,\al_n\}$ on $F$ which intersect transversely away from the crossings of $D$. 

\begin{lemma}\label{lem-NonIntersectingCurves}
If $\{\al_1,\ldots ,\al_n\}$ is a basis for $H_1(F;\ZZ)$ consisting of simple closed curves which intersect transversely away from the crossings of $D$, then $\cG_F$ is a sum of local contributions at the crossings of $D$.
\end{lemma}

\begin{proof}
It is enough to show that the pairing matrix in the basis $\{\al_1,\ldots ,\al_n\}$ depends only on the local contributions at crossings.

By the over-passing description of linking numbers (Table~\ref{tab:linking-number-crossings}), contributions to $\lk(\tau\al_i, \al_j)+\lk(\tau\al_j,\al_i)$ arise only where $\al_i$ intersects $\al_j$, and where the surface twists at the crossings of $D$.

Therefore, we only need to show that the contributions from the intersections of $\al_i$ with $\al_j$ -- all of which are away from the crossings of $D$ -- are zero. Indeed, the contributions to $\lk(\tau\al_i, \al_j)$ and $\lk(\tau\al_j,\al_i)$ are both $\pm 1$ and of opposite sign, see Figure~\ref{fig:GL-example}.  
\end{proof}

\begin{proof}[Proof of Theorem~\ref{thm-equiv}]
The black checkerboard surface $F_b(D)$ deformation retracts to the black Tait graph $G_b(D)$, thus $H_1(F_b(D);\ZZ) \cong H_1(G_b(D);\ZZ)$. We need to prove that the pairing is the same.

Choose a basis of oriented circuits for $\cF_b(D)$, that is, a set of minimal cycles $\{c_1,\ldots ,c_n\}$ in $G_b(D)$. The pairing $\langle c_i, c_j \rangle$ is the algebraic count of common edges between $c_i$ and $c_j$, where the sign depends on whether the orientations coincide. 

Let $\iota \co G_b(D) \hookrightarrow F_b(D)$ be the inclusion map, and let $\al_i$ be a simple closed curve representing the isomorphic image $\iota_*(c_i)$ in $H_1(F_b;\ZZ)$, for $i=1,\ldots,n$. It is possible to represent $\iota_*(c_i)$ by a simple closed curve, since $c_i$ is a simple closed walk in $G_b(D)$. By a local deformation, we can ensure that the curves in $\{\al_1,\dots,\al_n\}$ meet the conditions of Lemma~\ref{lem-NonIntersectingCurves}. 

Thus, $\cG_{F_b(D)}(\al_i,\al_j)$ is a sum of local contributions at the crossings of $D$, which arise when $\al_i$ and $\al_j$ traverse through the same crossing. Note that this happens exactly when $c_i$ and $c_j$ share an edge corresponding to that crossing, so we only need to verify that the contribution to $\cG_{F_b(D)}(\al_i,\al_j)$ is 1 whenever $\al_i$ and $\al_j$ traverse in the same direction, and $-1$ when they traverse in opposite directions. This is shown in Figure~\ref{fig:GL-example}.

\begin{figure}[ht!]
	\centering
	\includegraphics[height=22mm]{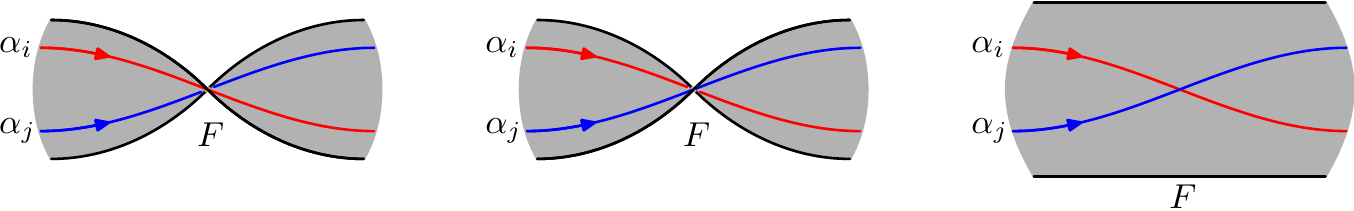}
	\caption{On left, $\al_i$ crosses over $\al_j$, with $\lk( \tau \al_i, \al_j) = 2$ and $\lk( \tau \al_j, \al_i) = 0$. In the middle, $\al_j$ crosses over $\al_i$, with $\lk( \tau \al_i, \al_j) = 0$ and $\lk( \tau \al_j, \al_i) = -2$. On the right, $\al_i$ and $\al_j$ intersect, with $\lk( \tau \al_i, \al_j)=1$ and $\lk( \tau \al_j, \al_i)=-1.$}
	\label{fig:GL-example}
\end{figure}

The same argument applies for the white checkerboard surface, but since the crossings are now of opposite type, the local contributions to $\cG_{F_w(D)}(\al, \be)$ have the opposite sign in each case. Hence, $\cG_{F_w(D)} \cong -\cF_w(D)$.
\end{proof}



\section{Main results}
\label{section-main}
In this section we present our main results. The first, in Section~\ref{subsection-Tait} is a new proof of the Tait conjectures for alternating links in thickened surfaces. It is modelled on \cite{Greene-2017} and uses the geometric characterization of alternating links in thickened surfaces \cite{Boden-Karimi-2023b}. 

In Section~\ref{subsection-GLLattices} we show that, for alternating links in thickened surfaces, the lattices of cuts and flows of the Tait graphs are link invariants. We also show that disc mutation does not change these lattices even for non-alternating links, and this is also true for the non-symmetric Gordon-Litherland form.  This implies in particular that, just as in the classical case, the $d$-invariants of alternating surface links are mutation invariants.

In Section~\ref{subsec:noncomplete} we show by example that the lattice of integer flows and the $d$-invariant are not complete invariants of the mutation class for links in thickened surfaces. We show that Gordon-Litherland form is in fact a strictly stronger mutation invariant. An outstanding question is to determine whether it is complete.

\subsection{Generalised Tait conjectures} \label{subsection-Tait}
In \cite{Greene-2017}, Greene gave a geometric proof of the first two Tait conjectures for classical links.
In this section, we adapt his argument and give a short proof of the first two Tait conjectures for virtual links. Our argument will use the characterisation of alternating links in thickened surfaces in \cite{Boden-Karimi-2023b}.

Other proofs of the generalised Tait conjectures for virtual knots have appeared recently. In \cite{Boden-Karimi-2019},  the Jones-Krushkal polynomial is used to establish an analogue of the Kauffman-Murasugi-Thistlethwaite result for links in thickened surfaces, and the first two Tait conjectures are derived from that. In \cite{Boden-Karimi-Sikora}, the skein bracket is used to give a strengthened statement of the  Tait conjectures. In \cite{Kindred-2022}, Kindred proves the virtual flyping theorem, which is the third Tait conjecture for virtual links, by extending his geometric proof of the flype theorem for classical links \cite{Kindred-2020}.

To prove the generalised Tait conjectures, recall one of the main results from \cite{Boden-Karimi-2023b}:

\begin{theorem}\label{thm-VirtualAltChar} 
Let $L$ be a link in $\Si \times I$, which bounds connected positive and negative definite spanning surfaces $F^+$ and $F^-$. Then $L$ is a non-split, alternating link, $\Si$ is a minimal genus surface to support $L$. Furthermore, there exists a reduced alternating diagram $D$ on $\Si$ whose black checkerboard surface $F_b$ is equivalent to $F^+$ and whose white checkerboard surface $F_w$ is equivalent to $F^-$.
\end{theorem}

The first statement of Theorem~\ref{thm-VirtualAltChar} is Theorem 1.1 of \cite{Boden-Karimi-2023b}. The second statement follows from the proof of that theorem, where $F^+$ and $F^-$ are glued along $L$ to construct a new model of $\Si$. 

\begin{theorem}[Generalised Tait Conjectures]\label{Thm-Tait}
    Any two connected, reduced alternating diagrams of the same non-split link $L \subseteq \Si \times I$ have the same crossing number and writhe.
\end{theorem}

\begin{proof}
    Let $L$ be a non-split link in $\Si \times I$, and assume that $D$ and $D'$ on $\Si$ are two connected, reduced alternating diagrams for $L$.

    After checkerboard colouring the two diagrams so that all crossings have type A, we have
    \begin{itemize}
    \item the black surface $F_b(D)$ is a positive definite spanning surface for $L$,
    \item the white surface $F_w(D')$ is a negative definite spanning surface for $L$.
    \end{itemize}

    By Theorem~\ref{thm-VirtualAltChar}, there exists a diagram $D''$ on $\Si$ whose black surface is equivalent to $F_b(D)$ and whose white surface is equivalent to $F_w(D')$.

    Since equivalent surfaces have isometric Gordon-Litherland pairings, and the Gordon-Litherland pairings on the checkerboard surfaces coincide with the lattices of integer flows of the corresponding Tait graphs, we have $\cF_b(D)\cong \cF_b(D'')$ and $\cF_w(D')\cong \cF_w(D'')$.

    By the Discrete Torelli Theorem~\ref{thm-DiscTorelli}, this implies that $G_b(D)$ and $G_b(D'')$ are 2-isomorphic as graphs, and in particular they must have the same number of edges. Thus, $D$ and $D''$ have the same number of crossings.

    In turn, the equivalence of $F_w(D')$ and $F_w(D'')$ implies, by the same argument, that $D'$ and $D''$ have the same number of crossings. Thus, all three diagrams have the same number of crossings, and this concludes the proof of the statement about crossing numbers.

    For the writhe, let $p(D)$ and $n(D)$ denote the number of positive and negative crossings in $D$, and similarly for $D', D''.$
    Since all crossings in $D$ and $D''$ are of type A, any type II crossing has positive writhe and incidence number $\eta_c=1$. Thus $\mu(F_b(D))=p(D)$ and $\mu(F_b(D''))=p(D'')$. 
    Using the fact that $\si(L,F_b(D)) =\si(L,F_b(D''))$ and appealing to
    equation \ref{eqn-linksignature}, we get
    $$\sig(\cG_{F_b(D)})-p(D)=\sig(\cG_{F_b(D'')})-p(D'').$$
   The two black surfaces $F_b(D)$ and $F_b(D'')$ are equivalent in the sense of Definition \ref{defn-ssequivalent}, so $\sig(\cG_{F_b(D)})=\sig(\cG_{F_b(D'')})$, and therefore $p(D)=p(D'')$. Since the crossing numbers of $D$ and $D''$ are equal, this also implies $n(D)=n(D')$.

    Repeating the same argument with the white surfaces $F_w(D')$ and $F_w(D'')$, we obtain $p(D')=p(D'')$ and $n(D')=n(D'')$. Since the writhe of $D$ is $p(D)-n(D),$ this shows that $D, D',$ and $D''$ must all have the same writhe.
\end{proof}


\subsection{Mutation and Gordon-Litherland lattices}\label{subsection-GLLattices}

In this section we prove that for reduced alternating link diagrams, the Gordon-Litherland linking form and pairing are link invariants that are unchanged by disc mutation.

Recall that for alternating link diagrams $D \subseteq \Si$,
our convention is to checkerboard colour $\Si \sm D$ so that all crossings of $D$ have type A. We use $F_b(D), F_w(D)$ to denote the black and white checkerboard surfaces according to this colouring.

\begin{theorem} \label{thm-isom}
Let $D$ and $D'$ be two reduced, alternating, cellular diagrams for the same non-split link $L \subseteq \Si \times I$.
Then there are isometries of the Gordon-Litherland pairings: $\cG_{F_b(D)} \cong \cG_{F_b(D')}$ and $\cG_{F_w(D)} \cong \cG_{F_w(D')}$.
\end{theorem}

\begin{proof}
The black checkerboard surface $F_b(D)$ of the first diagram is positive definite, and
the white checkerboard surface $F_w(D')$ of the second is negative definite. 
We apply Theorem \ref{thm-VirtualAltChar} with $F^+ = F_b(D)$ and $F^- = F_w(D')$ to produce a diagram $D''$ for $L$ whose black surface is equivalent to $F^+ = F_b(D)$
and whose white surface is equivalent to $F^- =F_w(D')$. 
Let $F_b(D'')$ and $F_w(D'')$ be the black and white checkerboard surfaces of the diagram $D''$.
Since equivalent surfaces have isometric Gordon-Litherland pairings, we have 
\begin{equation}\label{eq-GLforms}
\cG_{F_b(D)} \cong \cG_{F_b(D'')} \, \text{ and } \, \cG_{F_w(D')} \cong \cG_{F_w(D'')}.
\end{equation}

The Gordon-Litherland pairings on the checkerboard surfaces coincide with the lattices of integer
flows of the corresponding Tait graphs, we have $\cF_b(D) \cong \cF_b(D'')$. Since the equivalence also preserves the kernel of the inclusion map into $\Si$, we also have $\cF^\circ_b(D) \cong \cF^\circ_b(D'')$. Thus, the $d$-invariants of these lattices are isomorphic as well; for short we write $d_{\cF^\circ_b(D)}\cong d_{\cF^\circ_b(D'')}$. In turn, by Theorem \ref{thm-ChromDuality},
$d_{\cF^\circ_b(D)}\cong d_{\cF_w(D)}$, and $d_{\cF^\circ_b(D'')}\cong d_{\cF_w(D'')}$. Thus, 
$d_{\cF_w(D)}\cong d_{\cF_w(D'')}.$

In \cite[Theorem 1.3]{Greene-2011} Greene shows that the $d$-invariant of the lattice of integer flows of a graph uniquely determines the 2-isomorphism class of the graph, and in turn, the isometry class of the lattice of integer flows. Therefore,
$\cF_w(D) \cong \cF_w(D''),$
and by Theorem~\ref{thm-equiv}, and by Equation \eqref{eq-GLforms},
\[
\cG_{F_w(D)} \cong \cG_{F_w(D'')} \cong \cG_{F_w(D')}.
\]

The same argument from $\cF_w(D')\cong \cF_w(D'')$ shows
$\cG_{F_b(D)} \cong \cG_{F_b(D'')} \cong \cG_{F_b(D')}.$
This completes the proof.
 \end{proof}

\begin{corollary} \label{cor-invariance}
 The lattices of integer flows $\cF_b(D)$ and $\cF_w(D)$ are link invariants for non-split alternating links on surfaces.  
\end{corollary}

Next, we prove that $\cG_{F_b}$, and as a result $\cF_b$, are invariant under disc mutation, in fact, disc mutation leaves the pairing matrix unchanged. We prove this at the link diagram level for any cellular checkerboard coloured link diagram on $\Si$: it is not necessary to assume that the diagrams are alternating.
    
\begin{theorem}\label{thm-mutation1}
	Let $D, D' \subseteq \Si$ be two cellular checkerboard coloured link diagrams related by a sequence of disk mutations. Then $\cG_{F_b(D)} \cong \cG_{F_b(D')}$ and $\cG_{F_w(D)} \cong \cG_{F_w(D')}$.
\end{theorem}

The proof relies on the following technical lemma:

\begin{lemma}\label{lem-MutationBasis}
Given a link diagram $D\subseteq \Si$, checkerboard surface $F$, and mutation disc $B$, there exists a basis for $H_1(F;\ZZ)$ which satisfies the requirements of Lemma~\ref{lem-NonIntersectingCurves}, so that at most one of the simple closed curves meets $\partial B$. Moreover, if there is such a curve, it meets $\partial B$ in at most two points.
\end{lemma}

\begin{proof}
Let $X=N_{\epsilon} (F\cap B)$ and $Y=N_{\epsilon} (F \setminus B)$ be small neighbourhoods of the intersection of the mutation disc and its complement with $F$. By the Mayer-Vietoris Theorem, $H_1(X;\ZZ) \oplus H_1(Y;\ZZ)$ injects into $H_1(F;\ZZ)$. If $F\cap B$ has two connected components, then this is an isomorphism, and thus a basis can be chosen such that none of the simple closed curves intersect $\partial B$. If $F\cap B$ is connected, then $H_1(F;\ZZ)/\im\big(H_1(X;\ZZ) \oplus H_1(Y;\ZZ)\big)\cong \ZZ$, and thus a single generator of $H_1(F; \ZZ)$ must run through $B$ and thus meet $\partial D$. Further, this generator can be chosen so that it intersects $B$ in a single path.
\end{proof}

\begin{proof}[Proof of Theorem~\ref{thm-mutation1}]
	It is enough to prove this in the case where $D$ and $D'$ are related by a single mutation with mutation disc $B$. Note that disc mutation preserves the checkerboard colouring, and let $F'$ denote the mutant checkerboard surface.
 
 By Lemma \ref{lem-MutationBasis} we choose a basis of simple closed curves $\{\al_1,\ldots,\al_n\}$ for $H_1(F;\ZZ)$ whose only intersection points are away from the crossings of $D$, and at most one of the basis elements -- assume this is $\al_1$ -- meet the boundary of the mutation disk. That is, every $\al_i$ for $i\neq 1$ is either disjoint from $B$ or completely contained in it.

 Since at most $\al_1$ traverses $B$, and if so, then $|\al_1\cap \partial B|=2$, the disc mutation does not change the connectivity of the curves $\al_i$. We denote the mutant image of each $\al_i$ also by $\al_i$, and thus $\{\al_1,\dots, \al_n\}$ form a basis for the first homology of $H_1(F';\ZZ)$, satisfying the same conditions. 
 
 Notice that, under mutation, the type of the crossing ($A$ or $B$) does not change. The self-pairing of a curve is a count of the number of type $A$ crossings minus the number of type $B$ crossings. Since the crossing types do not change under mutation, it follows that $\cG_{F}(\al_i, \al_i) = $$\cG_{F'}(\al_i, \al_i)$ for all $i=1,\dots, n$.
 
\begin{figure}[ht!]
	\centering
	\includegraphics[height=16mm]{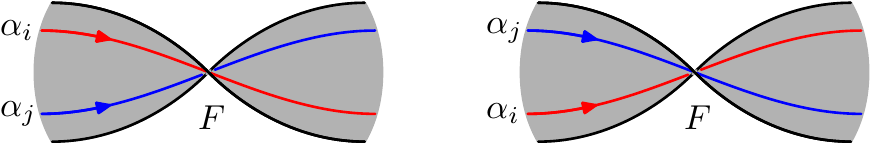}
	\caption{The effect of a mutation on the local contribution at a crossing.}
	\label{fig-MutationPairing}
\end{figure}

 Now consider two basis elements $\al_i, \al_j$. The pairing $\cG_{F}(\al_i, \al_j)$ is determined by local contributions at the crossings traversed by both $\al_i$ and $\al_j$. If none of those crossings are inside the mutation disk, then  their pairing is unchanged. Suppose  that $\al_i$ and $\al_j$ traverse a crossing contained in the mutation disk. It is a simple case analysis to check that the contribution to $\cG_F(\al_i, \al_j)$ is unchanged. On the left of Figure~\ref{fig-MutationPairing} (left) shows $\alpha_i$ passing over $\alpha_j$, thus 
 $\lk(\tau\al_i,\al_j) =2$ and $\lk(\tau\al_j,\al_i) =0$. On the right is the result of applying mutation given by a rotation of angle $\pi$ about a horizontal line. Now $\alpha_j$ passes over $\alpha_i$, and thus we  $\lk(\tau\al_i,\al_j) =0$ and $\lk(\tau\al_j,\al_i) =2.$ The other cases are verified similarly, and it follows that
  $\cG_{F}(\al_i, \al_j) = \cG_{F'}(\al_i, \al_j)$. This completes the proof.
\end{proof}

The next goal of this section is to prove that disc mutation also leaves the Gordon-Litherland linking form $\cL_F$ on a checkerboard surface $F$ unchanged. To do so, we use the following formula relating the linking form to the Gordon-Litherland pairing and intersection form.

\begin{lemma} \label{lemma-link}
    Let $D$ be a cellular, checkerboard coloured link diagram in $\Si$ with checkerboard surface $F\subseteq \Si \times I$, and let $p\co \Si\times I \to \Si$ denote the projection map.
	Let $\al$ and $\be$ be closed curves in $F$ with finitely many transverse intersections. Then the Gordon-Litherland linking form and pairing satisfy
 \begin{equation}\label{eq-PairingForm}
\cL_{F}(\al, \be) = \cG_{F}(\al, \be) + p_{*}(\al) \cdot p_{*}(\be).
\end{equation}
\end{lemma}

\begin{proof}
	 We first claim that algebraic intersection numbers and linking numbers satisfy
\begin{equation}\label{eq-IntLink}
 p_{*}(\al) \cdot p_{*}(\be) = \tfrac{1}{2}\big(\lk(\tau \al, \be) - \lk(\tau \be, \al)\big).
\end{equation}
By local deformations, we can ensure that $\al$ and $\be$ only intersect away from crossings. At these intersections it is easy to verify that the contribution to $p_{*}(\al) \cdot p_{*}(\be)$ and $\tfrac{1}{2}\big(\lk(\tau \al, \be) - \lk(\tau \be, \al)\big)$ are equal, cf.~Figure \ref{fig:GL-example} (right).
 
    Next, we examine the local contributions at the crossings of $D$. Note that since $\al$ and $\be$ do not intersect at the crossings, the projections $p(\al)$ and $p(\be)$ intersect transversely on $\Si$, as shown on the right of Figure~\ref{fig:GL-example}.
	At the crossings, it is elementary to verify that the local contributions to both sides of \eqref{eq-IntLink} are equal; see Figure~\ref{fig:int-example}.

\begin{figure}[hbt!]
	\centering
	\includegraphics[height=16mm]{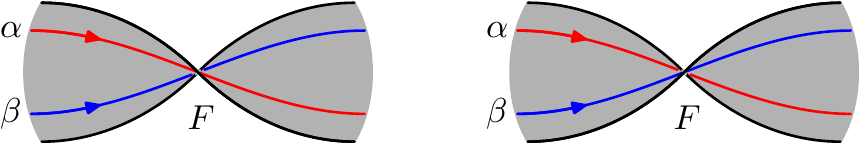}
	\caption{The local contributions are $\lk(\tau \al,\be)=2$ and $\lk(\tau \be, \al) =0$ on left, and $\lk(\tau \al,\be)=0$ and $\lk(\tau \be, \al) =-2$ on right. In either case, we have  $\lk(\tau \al,\be)-\lk(\tau \be, \al) =1=p_*(\al) \cdots p_*(\be).$ }
	\label{fig:int-example}
\end{figure}

Now apply Equation \eqref{eq-IntLink}   to complete the proof:
\begin{align*}
\cL_{F}(\al, \be) & =  \tfrac{1}{2}\left(\lk(\tau \al, \be) + \lk(\tau \be, \al)\right) + \tfrac{1}{2} \left(\lk(\tau \al, \be) - \lk(\tau \be, \al)\right)\\
& = \cG_{F}(\al, \be) + p_{*}(\al) \cdot p_{*}(\be). \qedhere
\end{align*}
\end{proof}

\begin{remark}
In \cite{Boden-Chrisman-Karimi}, the Gordon-Litherland pairing is defined by the formula
$\cG_{F}(\al,\be) = \lk(\tau(\al), \be)-p_{*}(\al) \cdot p_{*}(\be)$.
\end{remark}

The following corollary greatly simplifies the computation of the Gordon-Litherland linking form in examples (cf.~\cite[Remark 3.2]{Boden-Karimi-2023}).

\begin{corollary}\label{cor-OutsideKernel}
  Let $\iota\co F \to \Si \times I$ be the inclusion, inducing $\iota_{*}\co H_1(F;\ZZ) \to H_1(\Si \times I;\ZZ)$.
  If $\al \in \ker \iota_{*}$, then $\cL_{F}(\al, \be) = \cL_{F}(\be, \al)= \cL_{F}(\al, \be)$ for any $\be \in H_1(F;\ZZ)$. In particular, the restriction of $\cL_F$ to $\ker \iota_*$ is symmetric.
\end{corollary}

\begin{proof}
	Assume $\al \in \ker \iota_{*}$. Then $\al$ is null-homologous in $\Si \times I$, so $p_*(\al)$ is null-homologous in $\Si$. It follows that $p_{*}(\al) \cdot p_{*}(\be) = 0$. The proof is completed by applying Lemma~\ref{lemma-link}, and the observation that $\cG_{F}(\al, \be)$ is symmetric.
\end{proof}

Now we are able to show that $\cL_{F}$ is invariant under mutation.

\begin{theorem} \label{thm-linking-form-mutation}
	Let  $D,D' \subseteq \Si$ be two cellular link diagrams related by a finite sequence of disk mutations, with checkerboard surface $F$ for $D$ and mutant surface $F'$ for $D'$. Then $\cL_{F} \cong \cL_{F'}$.
\end{theorem}

\begin{proof}
	It is enough to prove the statement for a single mutation with mutation disc $B$. Choose a basis for $H_1(F;\ZZ)$ represented by simple closed curves $\{\al_1,\dots,\al_n\}$ as in Lemma~\ref{lem-MutationBasis}: the curves only intersect away from the crossings of $D$, and at most one curve, $\al_1$, intersects $\partial B$ in at most two points. Then the mutation does not change the connectivity of the curves, and we denote the mutant image of $\al_i$ by $\al'_i$.

    Furthermore, by Theorem~\ref{thm-mutation1} and Lemma~\ref{lemma-link}, it suffices to show that $p_{*}(\al_i) \cdot p_{*}(\al_j)= p_{*}(\al'_i) \cdot p_{*}(\al'_j)$. By Corollary~\ref{cor-OutsideKernel}, if $\al_i \in \ker \iota_*$, then the intersection pairing of $\al_i$ with any other curve is zero, hence, we only need to verify the statement for pairs of cycles not in $\ker \iota_*$.

    Assume that $\al_i$ and $\al_j$ are two curves with $\iota_*(\al_i)\neq 0 \neq \iota_*(\al_j)$. Thus, neither $\al_i$ nor $\al_j$ is contained in $B$. Since at most one curve traverses $B$, this implies that at least one of $\al_i$ and $\al_j$ is completely outside $B$. Therefore, $\al_i$ and $\al_j$ do not intersect inside $B$, which implies that the mutation does not change the intersection pairing: $p_{*}(\al_i) \cdot p_{*}(\al_j)= p_{*}(\al'_i) \cdot p_{*}(\al'_j)$. This concludes the proof. 
\end{proof}

This in particular implies that the link invariants derived from mock Seifert matrices are all invariant under mutation:
\begin{corollary}\label{cor-MutationInvariants}
    The mock Alexander polynomial, link determinant, signatures, and mock Levine-Tristram signatures are all invariants under disc mutation.
\end{corollary}


\subsection{Non-completeness} \label{subsec:noncomplete}
In this section, we show that the invariant $(\cF_b(L), \cF_w(L))$ is not a complete invariant of the mutation class of the alternating  link $L \subseteq \Si \times I$. We present examples of alternating knots in thickened surfaces whose Gordon-Litherland flow lattices are isomorphic but whose linking forms are not. 

The  calculation is carried out in terms of representative matrices. Given an alternating knot $K \subseteq \Si \times I$ with checkerboard colouring, let $F_b$ and $F_w$ denote the black and white checkerboard surfaces. We  use $A_b(K), A_w(K)$ to denote the mock Seifert matrices for the black and white surfaces,
and $S_b(K), S_w(K)$ for the Gordon-Litherland pairing matrices, which are the symmetrisations of $A_b(K), A_w(K)$.
We seek knots $K_{1}$ and $K_{2}$ whose pairing matrices are pairwise unimodular congruent but whose mock Seifert matrices are not. Two such knots would have equivalent lattice invariants, i.e., $(\cF_b(K_{1}), \cF_w(K_{1})) \cong (\cF_b(K_{2}), \cF_w(K_{2})),$ but inequivalent linking forms $(\cL_b(K_{1}), \cL_w(K_{1}))\not\cong (\cL_b(K_{2}), \cL_w(K_{2}))$. By Theorem \ref{thm-linking-form-mutation}, this implies that the knots are not equivalent under mutation. 

\begin{example}\label{ex-NonCompleteness}
Of the 92800 virtual knots up to 6 crossings, only 168 of them are alternating. A list of these virtual knots appears in Appendix B of \cite{Karimi}. For these virtual knots, we computed the above invariants and found multiple counterexample pairs. The full results of the computation are available at \cite{Lin-algo}. One of the simplest counterexamples is the pair of knots $K_1 =K_{6.90101}$ and $K_2 =K_{6.90124}$ shown in Figure~\ref{fig:counterexamples}. These knots have identical Gordon-Litherland pairing matrices, as shown in Equation~\eqref{eqn-GLMatrices}, but different mock Seifert matrices, shown in Equation \eqref{eqn-MSMatrices}.

\begin{figure}
	\centering
	\hspace*{\fill}
	\begin{subfigure}[b]{0.35 \textwidth}
		\centering
		\def\svgwidth{\columnwidth}
		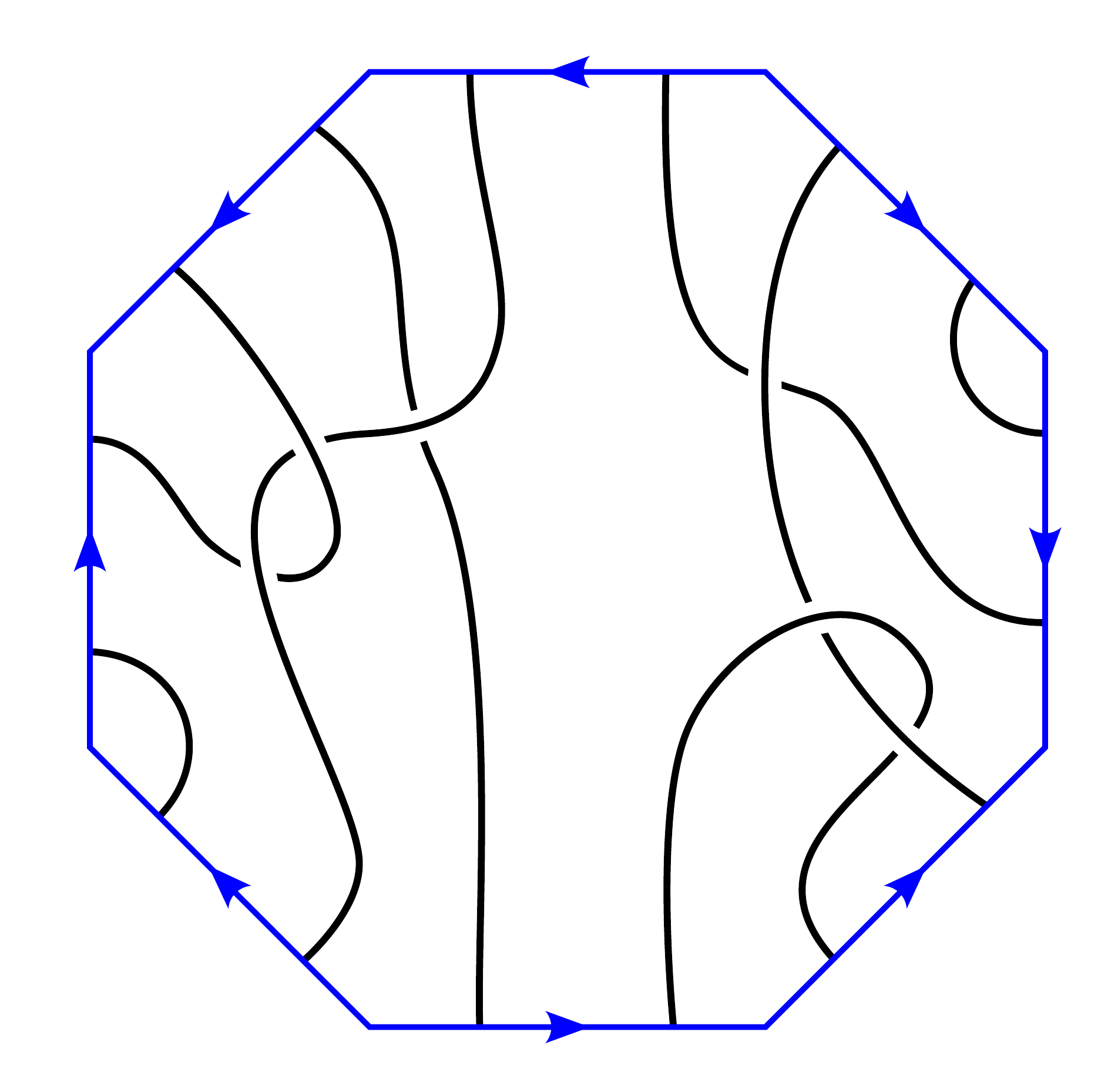
		\caption{Virtual knot $K_1=K_{6.90101}$}
		\label{fig:6-90101_vknot}
	\end{subfigure}
	\hspace*{\fill}	\hspace*{\fill}	\hspace*{\fill}
	\begin{subfigure}[b]{0.35 \textwidth}
		\centering
		\def\svgwidth{\columnwidth}
		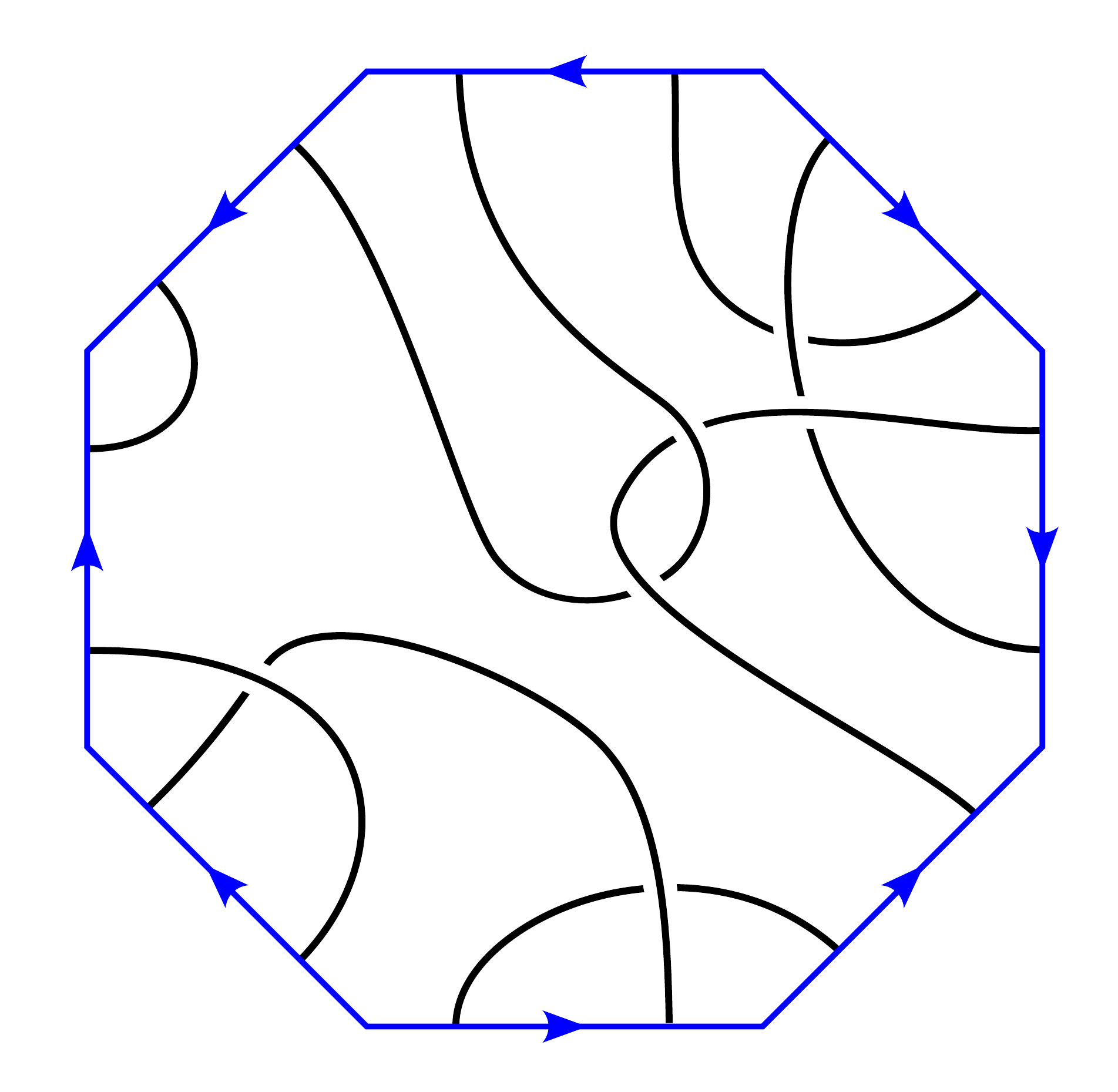
		\caption{Virtual knot $K_2=K_{6.90124}$}
		\label{fig:6-90124_vknot}
	\end{subfigure}
	\hspace*{\fill} 
	\caption{A pair of knots in a thickened surface of genus 2 providing a counterexample to the completeness of $(\cF_b(L), \cF_w(L))$ as a mutation invariant of the knot.}
	\label{fig:counterexamples}
\end{figure}

\begin{equation} \label{eqn-GLMatrices}
\begin{split}
	S_b(K_1) = S_b(K_2) &= \begin{bmatrix}
		2 & 0 & 0 & 0 \\
		0 & 1 & 0 & 0 \\
		0 & 0 & 2 & 0 \\
		0 & 0 & 0 & 1
	\end{bmatrix},\\
S_w(K_1) = S_w(K_2) &=
\begin{bmatrix}
	1  & 0 & 0 & 0 & 1  & -1 \\
	0  & 1 & 0 & 0 & 0  & 0  \\
	0  & 0 & 1 & 0 & 0  & 1  \\
	0  & 0 & 0 & 1 & 0  & 0  \\
	1  & 0 & 0 & 0 & 2  & -2 \\
	-1 & 0 & 1 & 0 & -2 & 4
\end{bmatrix},\\
\end{split}
\end{equation}

\begin{equation} \label{eqn-MSMatrices}
\begin{split}A_b(K_1)
	&= \begin{bmatrix}
	2 & -1 & 0 & 0\\
	1 & 1 & 0 & 0\\
	0 & 0 & 2 & -1\\
	0 & 0 & 1 & 1\\
\end{bmatrix}, \qquad
A_w(K_1) = \begin{bmatrix}
1  & 1 & 0  & 0 & 1   &-1\\
-1&  1  &0 &  0  &0   &0 \\
0&   0 & 1   &1&  0   &1 \\
0   &0  &-1 & 1 & 0   &0 \\
1  & 0&  0 &  0  &2  & -2\\
-1&  0 & 1&   0 & -2&  4 
\end{bmatrix},\\
A_b(K_2) &= \begin{bmatrix}
	2 & -1 & 0 & -1\\
	1 & 1 & 0 & 1\\
	0 & 0 & 2 & -1\\
	1 & -1 & 1 & 1\\
\end{bmatrix}, \qquad
A_w(K_2)
= \begin{bmatrix}
	1  & 1  & -1&  0  &1 &  -1\\	                     
	-1 & 1 &  1  & 0  &0  & 0 \\
	1  & -1  &1   &1  &0   &1 \\
	0  & 0  & -1&  1  &0   &0 \\
	1  & 0 &  0  & 0 & 2   &-2\\
	-1 & 0&   1   &0&  -2  &4 \\
\end{bmatrix}.
\end{split}
\end{equation}

In fact, the mock Seifert matrix pairs are not unimodular congruent;
as demonstrated by comparing their determinants:
\begin{align*}
	\det A_b(K_1) =\det A_w(K_1) & = 9,  \\
	\det A_b(K_2) =\det A_w(K_2) & = 15.
\end{align*}
 
The mock Alexander polynomials, also computed from the mock Seifert matrices, are also different:
\begin{align*}
	\De_{K_1,F_b} (t) & \doteq 9 t^{4} - 12 t^{3} + 22 t^{2} - 12 t + 9,  \\
    \De_{K_1,F_w} (t) & \doteq 9 t^{4} +12 t^{3} + 22 t^{2} - 12 t + 9, \\
	\De_{K_2,F_b}(t) & \doteq 15 t^{4} - 12 t^{3} + 10 t^{2} - 12 t + 15,\\
    \De_{K_2,F_w}(t) & \doteq 15 t^{4} + 12 t^{3} + 10 t^{2} + 12 t  + 15. 
\end{align*}
By Corollary~\ref{cor-MutationInvariants}, mutant alternating virtual knots have the same mock Alexander polynomials. Therefore, the knots $K_1$ and $K_2$ are not mutants of one another. 
\end{example}

\begin{corollary}\label{cor-NonCompleteness}
	For non-split alternating link diagrams on surfaces, the pair of lattices $(\cF_b(D), \cF_w(D))$
    of integer flows of the Tait graphs  is not a complete invariant of alternating links in thickened surfaces up to mutation.
\end{corollary}

We close this paper by recalling the statement of Kindred's virtual flype theorem \cite{Kindred-2022}, and using it to show that the the Gordon-Litherland linking form is an invariant for weakly prime alternating links in thickened surfaces.

Recall from \cite{Howie-Purcell} that a link diagram $D\subseteq \Si$ is \textit{weakly prime} if for any connected sum decomposition $(\Si, D)=(\Si, D_1)\#(S^2, D_2)$, either $D_2$ is the trivial diagram of the unknot, or $D_1$ is the trivial diagram of the unknot and $\Si=S^2$.

A \textit{flype move} is an operation on link diagrams does not alter the link type and preserves the alternating property, see Figure~\ref{fig-flype}. The Tait flype conjecture asserts that any two reduced alternating diagrams of the same link are equivalent by a finite sequence of flype moves. This was proved for classical links by Menasco and Thistlethwaite \cite{Menasco-Thistlethwaite}, and recently given a geometric proof by Kindred \cite{Kindred-2020}. The Tait conjectures have been proved for virtual links in the papers \cite{Boden-Karimi-2019, Boden-Karimi-Sikora, Kindred-2022}, and the next result is the statement of Kindred's virtual flyping theorem \cite{Kindred-2022}.

\begin{theorem}[Virtual Flyping Theorem]\label{thm-VFlype}
If $D$ and $D'$ are two weakly prime, cellularly embedded, alternating diagrams of a virtual link $L$, then $D$ and $D'$ are related by a sequence of flype moves.
\end{theorem}

We use this theorem to show that the Gordon-Litherland linking forms of the checkerboard surfaces are invariants of weakly prime alternating links in thickened surfaces. 

\begin{figure}
	\centering
 	\def\svgscale{0.38}
\begingroup%
  \makeatletter%
  \providecommand\color[2][]{%
    \errmessage{(Inkscape) Color is used for the text in Inkscape, but the package 'color.sty' is not loaded}%
    \renewcommand\color[2][]{}%
  }%
  \providecommand\transparent[1]{%
    \errmessage{(Inkscape) Transparency is used (non-zero) for the text in Inkscape, but the package 'transparent.sty' is not loaded}%
    \renewcommand\transparent[1]{}%
  }%
  \providecommand\rotatebox[2]{#2}%
  \newcommand*\fsize{\dimexpr\f@size pt\relax}%
  \newcommand*\lineheight[1]{\fontsize{\fsize}{#1\fsize}\selectfont}%
  \ifx\svgwidth\undefined%
    \setlength{\unitlength}{789.89519254bp}%
    \ifx\svgscale\undefined%
      \relax%
    \else%
      \setlength{\unitlength}{\unitlength * \real{\svgscale}}%
    \fi%
  \else%
    \setlength{\unitlength}{\svgwidth}%
  \fi%
  \global\let\svgwidth\undefined%
  \global\let\svgscale\undefined%
  \makeatother%
  \begin{picture}(1,0.12910689)%
    \lineheight{1}%
    \setlength\tabcolsep{0pt}%
    \put(0,0){\includegraphics[width=\unitlength,page=1]{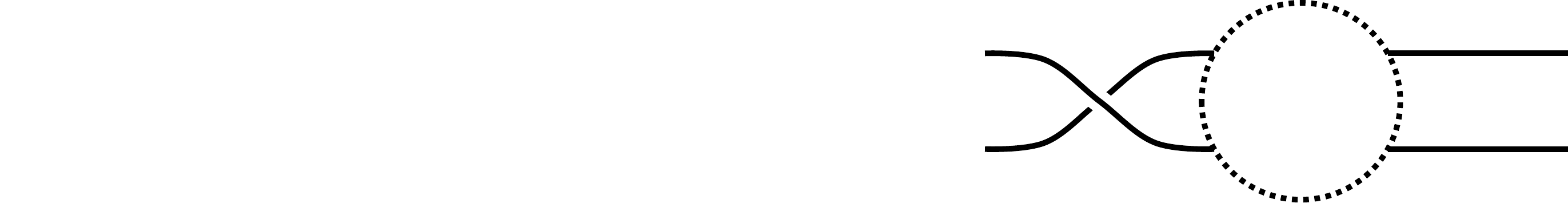}}%
    \put(0.17008238,0.04588961){\color[rgb]{0,0,0}\makebox(0,0)[t]{\lineheight{1.25}\smash{\begin{tabular}[t]{c}{\LARGE$T$}\end{tabular}}}}%
    \put(0.8297888,0.08620826){\color[rgb]{0,0,0}\rotatebox{-180}{\makebox(0,0)[t]{\lineheight{1.25}\smash{\begin{tabular}[t]{c}{\LARGE$T$}\end{tabular}}}}}%
    \put(0,0){\includegraphics[width=\unitlength,page=2]{flype.pdf}}%
    \put(0.49967665,0.08306175){\color[rgb]{0,0,0}\makebox(0,0)[t]{\lineheight{1.25}\smash{\begin{tabular}[t]{c}flype\end{tabular}}}}%
  \end{picture}%
\endgroup%

    \caption{The flype move.}\label{fig-flype}
\end{figure}

\begin{theorem} \label{thm-weakly-prime}
    Let $D$ and $D'$ be two weakly prime, cellularly embedded, alternating diagrams for the same link $L \subseteq \Si \times I$.
Then there are isomorphisms of the Gordon-Litherland linking forms: $\cL_{F_b(D)} \cong \cL_{F_b(D')}$ and $\cL_{F_w(D)} \cong \cL_{F_w(D')}$.
\end{theorem}

\begin{proof}
    By Theorem~\ref{thm-VFlype}, $D$ and $D'$ are related by a sequence of flype moves. It is enough to verify the statement for link diagrams $D$ and $D'$ that differ by a single flype move.

    Let $B\subseteq \Si$ be a disc which intersects $D$ in four points, and includes the tangle $T$ of Figure~\ref{fig-flype}, as well as the additional crossing which participates in the flype. Choose a basis of simple closed curves $\{\al_1,\ldots,\al_n\}$ for $H_1(F_b(D);\ZZ)$, satisfying the conditions of Lemma~\ref{lem-MutationBasis}: that is at most $\al_1$ traverses $B$, and the curves intersect transversely away from the crossings of $D$.

    By Lemma \ref{lemma-link}, we have $\cL_{F_b}(\al_i,\al_j)=\cG_{F_b}(\al_i,\al_j)+p_*(\al_i)\cdot p_*(\al_j)$. Furthermore, each $\al_i$ deformation retracts to a simple closed walk $a_i$ on the Tait graph $G_b$, and by Theorem~\ref{thm-equiv}, we have $\cG_{F_b}(\al_i,\al_j)=\langle a_i, a_j \rangle_{\cF_b}$ for all $i,j=1,\dots, n$. It is a short exercise to check that the flype move on the black Tait graph does not affect the flow pairings of the cycles $a_1,\ldots,a_n$. Thus, we only need to show that the intersection pairings remain the same, and this is true because there is only one generator outside the kernel of $\iota_*$ which may enter $B$.
\end{proof}

\begin{question}
    Are the isomorphism types of the Gordon-Litherland linking forms of the checkerboard surfaces invariants of alternating links in thickened surfaces?
\end{question}

Theorem \ref{thm-weakly-prime} answers this question for weakly prime, alternating link diagrams, and the question is whether it remains true for non-weakly prime, alternating link diagrams. Since they are invariant under mutation, it is natural to ask whether they give complete invariants.

\begin{question}
For non-split  alternating links $L$ in thickened surfaces, are the Gordon-Litherland linking forms complete invariants of the mutation class of $L$?
\end{question}

The last question concerns the lattice-theoretic formulas for the $d$-invariant of alternating virtual links  found in ~\S \ref{subsec-dInv}. 
For classical links, the $d$-invariant was originally defined geometrically as a correction term arising from the Heegaard-Floer homology of certain double branched covers in \cite{Ozsvath-Szabo-2003a}. In \cite{Ozsvath-Szabo-2003b, Ozsvath-Szabo-2005}, Ozsv\'{a}th and Szab\'{o} develop numerous applications of the $d$-invariant, including computational results for classical alternating links. Building on these ideas, Greene showed in \cite{Greene-2011} that the $d$-invariant of an alternating classical link depends only on the flow lattice of the associated Tait graph.
Our formulas in \S \ref{subsec-dInv} are the natural generalisation of Greene's approach, and they can be viewed as a step toward extending the $d$-invariant to all virtual knots and links.

\begin{question} \label{Q5-18}
Does the lattice-theoretic formulation of the $d$-invariant from~\S \ref{subsec-dInv} have a natural geometric interpretation? Does it extend to give a well-defined invariant on the class of all (checkerboard colourable) virtual knots and links?
\end{question}

In \cite{manolescu-owens}, the $d$-invariant is shown to descend to an invariant of knot concordance which, on alternating knots, coincides with the knot signature,  the Rasmussen $s$-invariant, and the $\tau$-invariant. (There is a vast literature on knot invariants arising from Heegaard Floer homology; for instance see  \cite{Greene-survey, Hom-survey, Hom-ICM} and the papers cited therein.)

If the $d$-invariant can be defined for all virtual knots, then as is the case for classical knots, it is likely that the extension is not unique. In order to answer Question \ref{Q5-18}, it would be most helpful to find a ``good'' extension of Heegaard-Floer homology, or one of its variants, to virtual links.  This interesting question is studied in \cite{JKO24}, where the authors propose extensions of Heegaard-Floer homology to links in thickened surfaces and study their behaviour under stable equivalence.

\vspace{3mm} \noindent
{\bf Acknowledgements.}
We would like to thank Homayun Karimi and Lorenzo Traldi for their valuable feedback, Marco Marengon for several insightful conversations, and the anonymous referee for their many helpful suggestions.

\bibliographystyle{alpha}

\end{document}